\documentclass[a4paper,10pt]{amsart}

\usepackage[top=2.5cm, bottom=2.5cm, left=2.5cm, right=2.5cm]{geometry}

\usepackage{amssymb}
\usepackage{ascmac}
\usepackage[dvipdfmx]{graphicx}
\usepackage{color}

\newtheorem{Def}{Definition}[section]
\newtheorem{Thm}[Def]{Theorem}
\newtheorem{Lem}[Def]{Lemma}
\newtheorem{Assumption}[Def]{Assumption}

\newtheorem{Rem}[Def]{Remark}
\newtheorem{Cor}[Def]{Corollary}
\newtheorem{Prop}[Def]{Proposition}

\newtheorem{Example}[Def]{Example}

\numberwithin{equation}{section}

\allowdisplaybreaks

\usepackage{ascmac}

\newcommand{\mca}{\mathcal{A}}

\newcommand{\mcc}{\mathcal{C}}

\newcommand{\mcf}{\mathcal{F}}
\newcommand{\mci}{\mathcal{I}}
\newcommand{\mcl}{\mathcal{L}}

\newcommand{\mbbf}{\mathbb{F}}
\newcommand{\mbbg}{\mathbb{G}}

\newcommand{\mbbn}{\mathbb{N}}

\newcommand{\mbbr}{\mathbb{R}}

\newcommand{\mbby}{\mathbb{Y}}

\newcommand{\al}{\alpha}
\newcommand{\del}{\delta}
\newcommand{\ep}{\epsilon} 
\newcommand{\vp}{\varphi}

\newcommand{\D}{\Delta}
\newcommand{\sig}{\sigma}
\newcommand{\Sig}{\Sigma}
\newcommand{\lam}{\lambda}

\newcommand{\gam}{\gamma}
\newcommand{\Gam}{\Gamma}
\newcommand{\p}{\partial}

\newcommand{\cil}{\xrightarrow{\mcl}} 
\newcommand{\cip}{\xrightarrow{P}} 
\newcommand{\argmax}{\mathop{\rm argmax}}

\def\nn{\nonumber}

\def\sumj{\sum_{j=1}^{n}}
\def\intj{\int_{t_{j-1}}^{t_j}}

\def\tz{\theta_{0}}
\def\tes{\hat{\theta}_{n}}
\def\aes{\hat{\alpha}_{n}}

\def\ges{\hat{\gamma}_{n}}

\title[Statistical inference for misspecified ergodic L\'{e}vy driven SDE models]
{Statistical inference for misspecified ergodic L\'{e}vy driven stochastic differential equation models}
\date{\today}
\keywords{L\'{e}vy driven stochastic differential equation, misspecified model, Gaussian quasi-likelihood estimation, extended Poisson equation, 
high-frequency sampling, stepwise estimation.}

\author{Yuma Uehara}
\address[Yuma Uehara, corresponding author]{The Institute of Statistical Mathematics, Japan, 10-3 Midori-cho, Tachikawa, Tokyo 190-8562, Japan}
\email{y-uehara@ism.ac.jp}

\begin{document}

\maketitle

\begin{abstract}
 We consider the estimation problem of misspecified ergodic L\'{e}vy driven stochastic differential equation models based on high-frequency samples.
We utilize a widely applicable and tractable Gaussian quasi-likelihood approach which focuses on mean and variance structure.
It is shown that the Gaussian quasi-likelihood estimators of the drift and scale parameters still satisfy polynomial type probability estimates and asymptotic normality at the same rate as the correctly specified case.
In their derivation process, the theory of extended Poisson equation for time-homogeneous Feller Markov processes plays an important role. 
Our result confirms the reliability of the Gaussian quasi-likelihood approach for SDE models.
\end{abstract}

\section{Introduction}

Recent development of measurement technique and computers enables us to get time series data observed at high frequency from financial and economic activities, physical and biological phenomena, and so on.
In analyzing them, we often encounter situations where they vividly demonstrate non-Gaussian behavior, and in such situations, statistical modeling by high-frequently observed diffusion type processes may bring us an inadequate result.
To better describe such non-Gaussianity, stochastic differential equations (SDEs) driven by L\'{e}vy processes serve as good candidate models.
For such reason, the estimation theory of L\'{e}vy driven SDE models based on high-frequency samples has been studied so far, for instance, the threshold based estimation for jump diffusion models by \cite{OgiYos11} and \cite{ShiYos06}, the least absolute deviation (LAD)-type estimation for L\'{e}vy driven Ornstein-Uhlenbeck models by \cite{Mas10}, the non-Gaussian stable quasi-likelihood estimation for locally stable driven SDE models by \cite{Mas16}, the least square estimation for small L\'{e}vy driven SDE models by \cite{LonMaShi16}, the Gaussian quasi-likelihood (GQL) for ergodic L\'{e}vy driven SDE models by \cite{Mas13-2} and \cite{MasUeh17-2}, and so on.
These are on parametric methods, and concurrently, nonparametric methods have been investigated, for example, the functional estimation and adaptive estimation for jump diffusion models by \cite{BanNgu03} and \cite{Sch14}, Nadaraya-Watson estimation for stable driven SDE models by \cite{Sch14}, and the Fourier based method for L\'{e}vy process and L\'{e}vy type model by \cite{BelRei15}, to mention few.

In statistical modeling, we always face the risk of model misspecification.
The statistical theory under model misspecification tells us how close an estimated model is to the data-generating model, and such interpretation is important, for example, in ensuring the reliability of estimation methods, and comparing candidate description models by information criterions.
Historically, following pioneering works by \cite{Ber66}, \cite{Hub67}, and \cite{Whi82}, the theory has been investigated up to the present for such reasons.
Especially about SDE models, for instance, \cite{McK84}, \cite{UchYos11} and \cite[Section 3]{Kut17} focus on misspecified diffusion models; \cite[Section 4]{Kut17} deal with the misspecification with respect to the intensity function of Poisson processes; \cite{RyaCheFra18} considers the situation where the given model is diffusion but the data-generating model has jumps.
However, the theory does not seem to be well developed in the context of L\'{e}vy driven SDE models whose coefficients take various non-linear form, and indeed the parametric methods introduced above are not discussed under model misspecification.

In this paper, the data-generating process $X$ which is defined on the complete filtered probability space $(\Omega,\mathcal{F}, (\mathcal{F}_t )_{t\in\mathbb{R}_+},P)$ is supposed to be the solution of the following L\'{e}vy driven SDE:
\begin{equation}\label{tSDE}
dX_t=A(X_t)dt+C(X_{t-})dZ_t,
\end{equation}
where:
\begin{itemize}
\item $Z$ is a one-dimensional c\`{a}dl\`{a}g L\'{e}vy process without Wiener part.
         It is independent of the initial variable $X_0$ and satisfies $E[Z_1]=0$, $Var[Z_1]=1$, and $E[|Z_1|^q]<\infty$ for all $q>0$;
\item The coefficients $A:\mbbr\mapsto\mbbr$ and $C:\mbbr\mapsto\mbbr$ are Lipschitz continuous;
\item $\mathcal{F}_t:=\sig(X_0)\vee\sig(Z_s; s\leq t)$.
\end{itemize}
We suppose that the discrete but high-frequency observations $(X_{t_0},\dots,X_{t_n})$ are obtained from $X$ in the so-called ``rapidly increasing experimental design", that is, $t_j\equiv t_j^n:=jh_n$, $T_n:=nh_n\to\infty$, and $nh_n^2\to0$.
For the observations $(X_{t_0},\dots,X_{t_n})$, we suppose that the following parametric one-dimensional SDE model is allocated:
\begin{equation}\label{pSDE}
dX_t=a(X_t,\al)dt+c(X_{t-},\gam)dZ_t,
\end{equation}
where the functional forms of the coefficients $a:\mbbr\times\Theta_\al\mapsto\mbbr$ and $c:\mbbr\times\Theta_\gam\mapsto\mbbr$ are supposed to be known except for a finite-dimensional unknown parameter $\theta:=(\gam,\al)$ being an element of the bounded convex domain $\Theta:=\Theta_\gam\times\Theta_\al\subset\mbbr^p$.
We note that the true coefficients $(C,A)(\cdot)$ may not belong to the parametric family $\{(c,a)(\cdot,\theta): \theta\in\Theta\}$, namely, the misspecification of the coefficient possibly occurs.
Hereinafter, the terminologies ``misspecified" and ``misspecification" will be used for the misspecification with respect to the coefficients unless otherwise mentioned.

To estimate an optimal parameter $\theta^\star$ of $\theta$, we utilize the GQL procedure used in \cite{MasUeh17-2}.
Concerning misspecified ergodic diffusion models, it is shown in \cite{UchYos11} that although the misspecification with respect to their diffusion term deteriorates the convergence rate of the scale (diffusion) parameter, the Gaussian quasi-maximum likelihood estimator (GQMLE) still has asymptotic normality.
We will show that asymptotic normality of the GQMLE holds in the misspecified ergodic L\'{e}vy driven SDE models as well.
To handle the misspecification effect, we will invoke the theory of the extended Poisson equation (EPE) for homogeneous Feller Markov processes established in \cite{VerKul11}.
Applying the result of \cite{VerKul11} for \eqref{tSDE}, the existence and weighted H\"{o}lder regularity of the solution of EPEs will be shown under a mighty mixing condition on $X$.
Building on the result and martingale representation theorem, we will be able to get the asymptotic normality of our estimator and its tail probability estimates under sufficient regularity and moment conditions on the ingredients of \eqref{tSDE} and \eqref{pSDE}.
We note that the absence of Wiener part in \eqref{tSDE} is essential while it is not in the correctly specified case, for more details, see Remark \ref{conj}.

\begin{table}[t]
\begin{center}
\caption{GQL approach for ergodic diffusion models and ergodic L\'{e}vy driven SDE models }
\begin{tabular}{cccc}
\hline
&&&\\[-3.5mm]
Model & \multicolumn{2}{c}{Rates of convergence}& Ref.\\
& drift  & scale\\ \hline 
correctly specified diffusion & $\sqrt{T_n}$& $\sqrt{n}$ &\cite{Kes97}, \cite{UchYos12}   \\[1.5mm]
misspecified diffusion & $\sqrt{T_n}$& $\sqrt{T_n}$ & \cite{UchYos11} \\ [1.5mm]
correctly specified L\'{e}vy driven SDE& $\sqrt{T_n}$& $\sqrt{T_n}$ &\cite{Mas13-1}, \cite{MasUeh17-2} \\ [1.5mm]
misspecified L\'{e}vy driven SDE &  $\sqrt{T_n}$ &$\sqrt{T_n}$ & this paper\\ \hline
\end{tabular}
\label{comp}
\end{center}
\end{table}

It will turn out that the convergence rate of the scale parameters is $\sqrt{T_n}$, and it is the same as the correctly specified case.
This is different from the diffusion case (cf. Table \ref{comp}).
Such difference may be caused from applying the GQL to non-Gaussian driving noises, that is, the efficiency loss of the GQMLE may occur even in the correctly specified case.
Indeed, the non-Gaussian stable quasi-likelihood is known to estimate the drift and scale parameters faster than the GQMLE in correctly specified locally $\beta$-stable driven SDE models (cf. \cite{Mas16}); each of their convergence rates are $\sqrt{n}h_n^{1-1/\beta}$ and $\sqrt{n}$, respectively.
Further, for correctly specified locally $\beta$-stable driven Ornstein-Uhlenbeck models, the LAD-type estimators of \cite{Mas10} tend to the true value at the speed of $\sqrt{n}h_n^{1-1/\beta}$ and it is also faster than that of the GQMLE.
However, in exchange for its efficiency, the GQL approach is worth considering by the following reasons:
\begin{itemize}
\item It does not include any special functions (e.g. Bessel function, Whittaker function, and so on), infinite expansion series and analytically unsolvable integrals, thus computation based on it is not relatively time-consuming.
\item It focuses only on the (conditional) mean and covariance structure, thus it does not need so much restriction on the driving noise and is robust against the noise structure. In other words, we can construct reasonable estimators of the drift and scale coefficients in the unified way if only the driving L\'{e}vy noise has moments of any order.
\end{itemize}
Our result ensures that even if the true coefficients are misspecified and take non-linear forms, the staged GQL estimation still works for L\'{e}vy driven SDE models and completely inherits its merit written in above.

The rest of this paper is organized as follows:
In Section \ref{ES}, we introduce assumptions and our estimation procedure.
Section \ref{Main} provides our main results in the following turn:
\begin{enumerate}
\item the tail probability estimates of the GQMLE (Theorem \ref{TPE});
\item the existence and weighted H\"{o}lder regularity of the solution of EPEs for L\'{e}vy driven SDEs (Proposition \ref{EPE});
\item the asymptotic normality of the GQMLE at $\sqrt{T_n}$-rate (Theorem \ref{AN}).
\end{enumerate}
A simple numerical experiment is presented in Section \ref{NE}.
We give all proofs of our results in Section \ref{AP}.




\section{Assumptions and Estimation scheme}\label{ES}
For notational convenience, we will hereafter use the following manners without any mention:
\begin{itemize} 
\item $\eta$ stands for the law of $X_0$.
\item $\bar{S}$ denotes the closure of any set $S$.
\item $\nu_0$ represents the L\'{e}vy measure of $Z$.
\item We write $x^{\otimes2}=x^\top x$ for any vector $x$.
\item $P_t(x,\cdot)$ denotes the transition probability of $X$. 
\item $\p_x$ is referred to as a differential operator for any variable $x$.
\item $E_j[\cdot]$ denotes the conditional expectation with respect to $\mcf_{t_j}$.
\item We write $Y_j=Y_{t_j}$ and $\D_j Y=Y_j-Y_{j-1}$ for any stochastic process $Y$.  
\item $x_n\lesssim y_n$ implies that there exists a positive constant $C$ being independent of $n$ satisfying $x_n\leq Cy_n$ for all large enough $n$.
\item For any matrix valued function $f$ on $\mbbr\times\Theta$, we write $f_s(\theta)=f(X_s,\theta)$; especially we write $f_j(\theta)=f(X_j,\theta)$.
\item Given a function $\rho:\mathbb{R}\to\mathbb{R}^+$ and a signed measure $m$ on a one-dimensional Borel space, we write
\begin{equation}\nn
||m||_\rho=\sup\left\{|m(f)|:\mbox{$f$ is $\mathbb{R}$-valued, $m$-measurable and satisfies $|f|\leq\rho$}\right\}.
\end{equation}
\end{itemize}

To derive our asymptotic results, we introduce some assumptions with some technical comments.
Most of them are almost the same as in \cite{Mas13-1}, \cite{MasUeh17-2}, and \cite{MasUeh17-1}, except for Assumption \ref{Moments}-(2).

\begin{Assumption}\label{Moments}
\begin{enumerate}
\item $E[Z_1]=0$, $Var[Z_1]=1$, and $E[|Z_1|^q]<\infty$ for all $q>0$.
\item The Blumenthal-Getoor index (BG-index) of $Z$ is smaller than 2, that is, 
\begin{equation*}
\beta:=\inf_\gam\left\{\gam\geq0: \int_{|z|\leq1}|z|^\gam\nu_0(dz)<\infty\right\}<2.
\end{equation*}
\end{enumerate}
\end{Assumption}

From \cite[Theorem 25.3]{Sat99}, it is easy to observe that Assumption \ref{Moments} holds if the L\'{e}vy measure $\nu_0$ admits a density $g$ with respect to Lebesgue measure satisfying that 
$g(z)=O(|z|^{-2-\del})$ as $|z|\to0$ for some $\del\in(0,1)$, and that there exist positive constants $K_0$, $K_1$ and $K_2$ such that
\begin{equation*}
g(z)\leq K_0(1+|z|^{K_1})e^{-|z|^{K_2}},
\end{equation*}
for all large enough $|z|$.
Via standardization, various L\'{e}vy processes fulfill them, for example, bilateral gamma process, normal tempered stable process, normal inverse Gaussian process, and variance gamma process.

In the derivation of the asymptotic normality of our estimator, we will evaluate the small time $L_{2-\ep}$-moment of $X$ for some $\ep>0$ (cf. Lemma \ref{FEV}) to handle the solution of extended Poisson equations which are essential to deal with the misspecification effect; thus the additional condition Assumption \ref{Moments}-(2) is imposed.


\begin{Assumption}\label{Smoothness}
\begin{enumerate}
\item The coefficients $A$ and $C$ are Lipschitz continuous and twice differentiable, and their first and second derivatives are of at most polynomial growth.
\item The drift coefficient $a(\cdot,\al^\star)$ and scale coefficient $c(\cdot,\gam^\star)$ are Lipschitz continuous, and $c(x,\gam)\neq0$ for every $(x,\gam)$.
\item For each $i \in \left\{0,1\right\}$ and $k \in \left\{0,\dots,5\right\}$, the following conditions hold:
\begin{itemize}
\item The coefficients $a$ and $c$ admit extension in $\mathcal{C}(\mathbb{R}\times\bar{\Theta})$ and have the partial derivatives $(\partial_x^i \partial_\alpha^k a, \partial_x^i \partial_\gamma^k c)$ possessing extension in $\mathcal{C}(\mathbb{R}\times\bar{\Theta})$.
\item There exists nonnegative constant $C_{(i,k)}$ satisfying
\begin{equation}\label{polynomial}
\sup_{(x,\alpha,\gamma) \in \mathbb{R} \times \bar{\Theta}_\alpha \times \bar{\Theta}_\gamma}\frac{1}{1+|x|^{C_{(i,k)}}}\left\{|\partial_x^i\partial_\alpha^ka(x,\alpha)|+|\partial_x^i\partial_\gamma^kc(x,\gamma)|+|c^{-1}(x,\gamma)|\right\}<\infty.
\end{equation}
\end{itemize}
\end{enumerate}
\end{Assumption}
We note that the first part of Assumption \ref{Moments} and Assumption \ref{Smoothness} ensures the existence of a unique c\`{a}dl\`{a}g adapted strong solution of SDE \eqref{tSDE} (cf. \cite[Theorem 6.2.3 and Theorem 6.2.9]{App09}), that is, there exists a measurable function $g$ such that $X=g(X_0,Z)$.

\begin{Assumption}\label{Stability}
\begin{enumerate}
\item
There exists a probability measure $\pi_0$ such that for every $q>0$, we can find constants $a>0$ and $C_q>0$ for which 
\begin{equation}\label{Ergodicity}
\sup_{t\in\mathbb{R}_{+}} \exp(at) ||P_t(x,\cdot)-\pi_0(\cdot)||_{h_q} \leq C_qh_q(x),
\end{equation}
for any $x\in\mathbb{R}$ where $h_q(x):=1+|x|^q$.
\item 
For any $q>0$,  we have
\begin{equation}
\sup_{t\in\mathbb{R}_{+}}E[|X_t|^q]<\infty. 
\end{equation}
\end{enumerate}
\end{Assumption}

The former property of this assumption is so-called ``$f$-exponentially ergodic" property (cf. \cite{MeyTwe93-3}), and putting together with the latter condition and the argument in \cite[Lemma 8]{Kes97} and \cite[Lemma 4.3]{Mas13-1}, it ensures the ergodic theorem, and its moment bound: for any $f$ being differentiable with derivatives of polynomial growth, we have
\begin{equation}\label{ethm}
\frac{1}{n}\sumj f_{j-1}\cip \int_\mbbr f(x)\pi_0(dx),
\end{equation}
and for any positive constant $K>0$,
\begin{equation}\label{mboun}
E\left[\left|\sqrt{T_n}\left(\frac{1}{n}\sumj f_{j-1}-\int_\mbbr f(x)\pi_0(dx)\right)\right|^K\right]<\infty.
\end{equation}
The first convergence in probability \eqref{ethm} is a standard condition assumed in the statistical theory of the ergodic processes, while the second moment bound \eqref{mboun} is not and is relatively strong.
It will be utilized for evaluating the tail probability of the staged GQL random field introduced later.
Such evaluation gives the tail probability estimates of our estimator (Theorem \ref{TPE}), and in turn, the convergence of moments of any order for it (Remark \ref{Mcon}).

The sufficient conditions of the ``$f$-exponentially ergodic" property for \eqref{tSDE} are investigated by many papers such as \cite{Kul09}, \cite{Mas07}, and \cite{Mas13-1}.
Among them, we introduce a handy one given in \cite[Section 5]{Mas13-1} in the following:
\medskip

\textbf{Condition 1} The coefficients $A$ and $C$ are of class $\mcc^1$, and globally Lipschitz, and the scale coefficient $C$ is bounded.

\smallskip
\textbf{Condition 2} The drift coefficient $A$ satisfies 
\begin{equation}
\limsup_{|x|\to\infty} \text{sgn}(x)A(x)<0,
\end{equation} 
and the scale coefficient $C(x)\neq 0$, for every $x\in\mbbr$.

\smallskip
\textbf{Condition 3} The L\'{e}vy measure $\nu_0$ of $Z$ can be decomposed as: $\nu_0=\nu_0^\star+\nu_0^\sharp$ for the two L\'{e}vy measure, where the restriction of $\nu_0^\star$ to some open set of the form $(-\ep,0)\cup (0,\ep)$ with some $\ep>0$ admits a continuously differentiable positive density $g^\star$.

\smallskip
\textbf{Condition 4} $E[Z_t]=0$ and $\int\exp(q|z|)\nu_0(dz)<\infty$ for some $q>0$.

\medskip
Under Condition 1-Condition 4, Assumption \ref{Stability} holds true and for its proof, see \cite[Proposition 5.4]{Mas13-1}.
We here note that this sufficient condition still allows the nonlinearity of the coefficients.
For example, given a L\'{e}vy process $Z$ fulfilling Condition 3 and Condition 4, the following SDEs satisfy Condition 1, Condition 2, and Assumption \ref{Smoothness}-(1):
\begin{enumerate}
\item
$dX_t=-X_tdt+\sqrt{1+X_{t-}^2}dZ_t$;
\item $\displaystyle{dX_t=-\frac{X_t}{\sqrt{1+X_{t}^2}}dt+dZ_t}$;
\item $\displaystyle{dX_t=-(X_t+2\sin X_t)dt+\frac{3+X_{t-}^2}{1+X_{t-}^2}dZ_t}$.
\end{enumerate}

\medskip
We introduce a $p\times p$-matrix $\Gam:=\begin{pmatrix}\Gam_\gam&O\\ \Gam_{\al\gam}&\Gam_\al\end{pmatrix}$ whose components are defined by:
\begin{align*}
&\Gam_\gam:=2\int_\mbbr\frac{\p_\gam^{\otimes2}c(x,\gam^\star)c(x,\gam^\star)-(\p_\gam c(x,\gam^\star))^{\otimes2}}{c^4(x,\gam^\star)}(C^2(x)-c^2(x,\gam^\star))\pi_0(dx)\\
&\qquad -4\int_\mbbr \frac{(\p_\gam c(x,\gam^\star))^{\otimes2}}{c^4(x,\gam^\star)}C^2(x)\pi_0(dx),\\
&\Gam_{\al\gam}:=2\int_\mbbr \p_\al a(x,\al^\star)\p_\gam^\top c^{-2}(x,\gam^\star)(a(x,\al^\star)-A(x))  \pi_0(dx),\\
&\Gam_\al:=-2\int_\mbbr\frac{\p_\al^{\otimes2} a(x,\al^\star)}{c^2(x,\gam^\star)}(a(x,\al^\star)-A(x))\pi_0(dx)-2\int_\mbbr\frac{(\p_\al a(x,\al^\star))^{\otimes2}}{c^2(x,\gam^\star)}\pi_0(dx).
\end{align*}
\begin{Assumption}\label{nd}
$\Gam$ is invertible.
\end{Assumption}
\medskip
We define an optimal parameter $\theta^\star:=(\gam^\star,\al^\star)$ of $\theta$ by 
\begin{equation*}
\gam^\star\in\argmax_{\gam\in\bar{\Theta}_\gam}\mbbg_1(\gam),\quad \al^\star\in\argmax_{\al\in\bar{\Theta}_\al}\mbbg_2(\al),
\end{equation*}
where $\mbbg_1: \Theta_\gam\mapsto \mbbr$ and $\mbbg_{2}:\Theta_\al\mapsto\mbbr$ are defined as follows:
\begin{align}
&\mbbg_1(\gam)=-\int_\mbbr \left(\log c^2(x,\gam)+\frac{C^2(x)}{c^2(x,\gam)}\right)\pi_0(dx),\label{KL1}\\
&\mbbg_2(\al)=-\int_\mbbr c(x,\gam^\star)^{-2}(A(x)-a(x,\al))^2\pi_0(dx). \label{KL2}
\end{align}
Note that since we impose the extension condition in Assumption \ref{Smoothness}, $\mbby(\theta):=(\mbby_1(\gam),\mbby_2(\al))$ admit extension in $\mcc(\bar{\Theta})$ as well.
Recall that the parameter space $\Theta$ is supposed to be a bounded convex domain.
We assume the following identifiability condition for $\mbbg_1(\gam)$ and $\mbbg_2(\al)$:
\begin{Assumption}\label{Identifiability}
$\theta^\star\in\Theta$,
and there exist positive constants $\chi_\gam$ and $\chi_\al$ such that for all $(\gam,\al)\in\Theta$,
\begin{align}
&\label{idg}\mbby_1(\gam):=\mbbg_1(\gam)-\mbbg_1(\gam^\star)\leq-\chi_\gam|\gam-\gam^\star|^2,\\
&\label{ida}\mbby_2(\al):=\mbbg_2(\al)-\mbbg_2(\al^\star)\leq-\chi_\al|\al-\al^\star|^2.
\end{align}
\end{Assumption}

\eqref{idg} and \eqref{ida} ensure the separability of the models which will also be used for the tail probability estimates of the staged GQL random fields, and the next remark provides a sufficient and non-stringent condition for them.

\begin{Rem}
If the optimal parameter $\theta^\star\in\Theta$ is unique, and $-\Gam_\gam$ and $-\Gam_\al$ are positive definite, \eqref{idg} and \eqref{ida} hold true for all $(\gam,\al)\in\Theta$ under Assumption \ref{Smoothness}-\ref{Stability}.
Let $\mci_1:\Theta_\gam\mapsto\mbbr^{p_\gam}\otimes\mbbr^{p_\gam}$ and $\mci_2:\Theta_\al\mapsto\mbbr^{p_\al}\otimes\mbbr^{p_\al}$ be
\begin{align*}
&\mci_1(\gam)=2\int_\mbbr\frac{\p_\gam^{\otimes2}c(x,\gam)c(x,\gam)-(\p_\gam c(x,\gam))^{\otimes2}}{c^4(x,\gam)}(C^2(x)-c^2(x,\gam))\pi_0(dx)-4\int_\mbbr \frac{(\p_\gam c(x,\gam))^{\otimes2}}{c^4(x,\gam)}C^2(x)\pi_0(dx),\\
&\mci_2(\al)=-2\int_\mbbr\frac{\p_\al^{\otimes2} a(x,\al)}{c^2(x,\gam^\star)}(a(x,\al)-A(x))\pi_0(dx)-2\int_\mbbr\frac{(\p_\al a(x,\al))^{\otimes2}}{c^2(x,\gam^\star)}\pi_0(dx).
\end{align*}
From Assumption \ref{Smoothness} and \ref{Stability}, the Lebesgue dominated convergence theorem implies that these functions are continuous.
Thus, for sufficiently small $\ep>0$, we can pick a positive constant $\del$ satisfying $U_\del(\gam^\star)\subset\Theta_\gam$ and $\inf_{\gam\in U_\del(\gam^\star)}\lam_{\text{min}}(-\mci_1(\gam))>\ep$ where $U_\del(\gam^\star)$ denotes the open ball of radius $\del$ centered at $\gam^\star$, and $\lam_{\text{min}}(-\mci_1(\gam))$ is a minimum eigenvalue of $- \mci_1(\gam)$.
Then, for every $\gam\in U_\del(\gam^\star)$, we have $\mbby_1(\gam)<-\ep|\gam-\gam^\star|^2$ by Taylor's formula.
Concerning $\gam\in \Theta_\gam\setminus U_\del(\gam^\star)$, it follows that
\begin{align*}
\mbby_1(\gam)<-\frac{\mbbg_1(\gam^\star)-\sup_{\gam\in\Theta_\gam\setminus U_\del(\gam^\star)} \mbbg_1(\gam)}{\sup_{\gam_1,\gam_2\in\Theta_\gam\setminus U_\del(\gam^\star)} |\gam_1-\gam_2|^2} |\gam-\gam^\star|^2.
\end{align*}
Hence \eqref{idg} holds true for all $\gam\in\Theta_\gam$ with 
\begin{equation*}
\chi_\gam=\ep\vee\frac{\mbbg_1(\gam^\star)-\sup_{\gam\in\Theta_\gam\setminus U_\del(\gam^\star)} \mbbg_1(\gam)}{\sup_{\gam_1,\gam_2\in\Theta_\gam\setminus U_\del(\gam^\star)} |\gam_1-\gam_2|^2}.
\end{equation*}
\eqref{ida} can be shown as well.
\end{Rem}

\medskip

From now on, we mention our estimation scheme.
Recall that we assume that the observation $(X_{t_0},\dots,X_{t_n})$ is obtained from $X$ with $t_j\equiv t_j^n:=jh_n$, $T_n:=nh_n\to\infty$, and $nh_n^2\to0$.
We define our staged GQMLE $\tes:=(\ges,\aes)$ in the following manner:
\begin{enumerate}
\item {\it Drift-free estimation of $\gamma$.} \label{step1}
Define the Maximizing-type estimator (so-called $M$-estimator) $\ges$ by
\begin{equation}
\ges\in\argmax_{\gam\in\bar{\Theta}_\gam}\mbbg_{1,n}(\gam),
\nonumber
\end{equation}
for the $\mbbr$-valued random function
\begin{equation*}
\mbbg_{1,n}(\gam):=-\frac{1}{T_n}\sumj \left\{h_n\log c^2_{j-1}(\gam)+\frac{(\D_j X)^2}{c^2_{j-1}(\gam)}\right\}.
\end{equation*}

\item {\it Weighted least square estimation of $\al$.} 
Define the least square type estimator $\aes$ by
\begin{equation}
\aes\in\argmax_{\al\in\bar{\Theta}_\al}\mbbg_{2,n}(\al),
\nonumber
\end{equation}
for the $\mbbr$-valued random function
\begin{equation}
\mbbg_{2,n}(\al):=-\frac{1}{T_n}\sumj \frac{(\D_j X-h_na_{j-1}(\al))^2}{h_nc^2_{j-1}(\ges)}.
\nonumber
\end{equation}
\end{enumerate}


\begin{Rem}\label{decay}
Although our estimation method ignores the drift term in the first stage, the effect of it asymptotically vanishes.
This is because the scale term dominates the small time behavior of $X$ in $L_2$-sense. 
Specifically, we can derive
\begin{equation*}
E_{j-1}\left[\left(\intj f_{s-}dZ_s\right)^2\right]\lesssim h_n f^2_{j-1}, \quad E_{j-1}\left[\left(\intj g_{s}ds\right)^2\right]\lesssim h_n^2 g^2_{j-1},
\end{equation*}
for suitable functions $f$ and $g$.
Indeed, it has already been shown that the asymptotic behavior of the scale estimator constructed by our manner is the same as the conventional GQL estimator in the case of correctly specified ergodic diffusion models (cf. \cite{UchYos12}) and ergodic L\'{e}vy driven SDE models (cf. \cite{MasUeh17-2}).
Such ignorance should be helpful in reducing the number of simultaneous optimization parameters, thus our estimator is expected to numerically be more stabilized and their calculation should be less time-consuming. 
Moreover, by choosing appropriate functional forms, each estimation stage is reduced to a convex optimization problem.
For example, if $a(\cdot, \cdot)$ and $c(\cdot, \cdot)$ are linear and log-linear with respect to parameters, respectively, then the above argument holds.
As for other candidates of their functional form and details, see \cite[Example 3.8]{MasUeh17-2}.
\end{Rem}

\begin{Rem}\label{optim}
We defined the optimal parameter of $\theta$ as the argmax point of $\mbbg_1(\gam)$ and $\mbbg_2(\al)$ and, the two functions are the probability limit of the Gaussian quasi-likelihoods $\mbbg_{1,n}(\gam)$ and $\mbbg_{2,n}(\al)$, respectively.
Thus, $-\mbbg_1(\gam)$ and $-\mbbg_2(\al)$ can be regarded as Kullback-Leibler (KL) divergence like quantities between the data-generating model and the parametric model $dX_t=a(X_t,\al)dt+c(X_{t-},\gam)dZ_t$. 
Here we first consider the correctly specified case, that is, there exists an element $\tz:=(\gam_0,\al_0)\in\Theta$ such that $C(x)= c(x,\gam_0)$ and $A(x)= a(x,\al_0)$ for $\pi_0$ a.s. $x$.
Fix a positive constant $b>0$.
Then, it can readily be checked that for all $x>0$, $\log x+\frac{b}{x}\geq \log b+1$, and that both sides are equivalent when $x=b$.
Hence, by Assumption \ref{Identifiability}, $\argmax_{\gam\in\bar{\Theta}_\gam}\mbbg_1(\gam)$ and $\argmax_{\al\in\bar{\Theta}_\al}\mbbg_2(\al)$ coincide with $\gam_0$ and $\al_0$, respectively.
In other words, this asserts that the data-generating model certainly attain the minimization of $-\mbbg_1(\gam)$ and $-\mbbg_2(\al)$.
By taking these insight into consideration, we can intuitively interpret the optimality of $\theta^\star$ as the parameter value which yields the closest model to the data-generating model measured by the Kullback-Leibler (KL) divergence like quantities $-\mbbg_1(\gam)$ and $-\mbbg_2(\al)$. 
\end{Rem}

\section{Main results}\label{Main}
In this section, we state our main results only for the fully misspecified case, that is, both of the true coefficients $C$ and $A$ do not belong to the parametric family $\{(c,a)(\cdot,\theta): \theta\in\Theta\}$.
Concerning the partly misspecified case (i.e. either of $C$ and $A$ is correctly specified), similar results can be derived just as the corollaries (see, Remark \ref{partly}).
All of their proofs will be given in Appendix.

The first result provides the tail probability estimates of the normalized $\tes$ which is theoretically essential such as in the deviation of an information criterion, residual analysis, and the measurement of $L_q$-prediction error.

\begin{Thm}\label{TPE}
Suppose that Assumptions \ref{Moments}-\ref{nd} hold.
Then, for any $L>0$ and $r>0$, there exists a positive constant $C_L$ such that
 \begin{equation}\label{eq: TPE}
\sup_{n\in\mbbn} P\left(\left|\sqrt{T_n}(\tes-\theta^\star)\right|>r\right)\leq \frac{C_L}{r^L}.
 \end{equation}
\end{Thm}

In the correctly specified case, such estimates are already shown in \cite{MasUeh17-2} under a sufficient moment and regularity conditions, and strong identifiability conditions, and this theorem extends the results to the misspecified case.

Before we state the asymptotic normality of $\tes$, we roughly explain how the misspecification effect arises in its derivation process, and introduce the useful tool to deal with it.
Except for $o_p(1)$ term, each scaled quasi-score function can be decomposed as:
\begin{equation}\label{me}
\mbox{(scaled quasi-score function)} = (\mbox{stochastic integral})+(\mbox{misspecification effect term}),
\end{equation}
where the misspecification effect term is expressed as:
\begin{equation}\label{met}
\sqrt{\frac{h_n}{n}}\sumj g_{j-1}(\theta^\star)=\frac{1}{\sqrt{T_n}}\int_0^{T_n} g_s(\theta^\star)ds+o_p(1),
\end{equation}
with a specific measurable function $g$ satisfying $\pi_0(g)=0$.
The celebrated CLT-type theorems for such single functional integration of Markov processes have been reported in many literatures, for example, \cite[Theorem 2.1]{Bha82}, \cite[Theorem V\hspace{-.1em}I\hspace{-.1em}I\hspace{-.1em}I  3.65]{JacShi03}, \cite[Theorem 2.1]{KomWal12}, \cite[Corollary 4.1]{VerKul13}, and the references therein.
However the combination with the stochastic integral makes it difficult to clarify the asymptotic behavior of the left-hand-side.
To handle this difficulty, we invoke the concept of the extended Poisson equation (EPE) introduced in \cite{VerKul11}:
\begin{Def}\cite[Definition 2.1]{VerKul11}
We say that a measurable function $f:\mbbr\to\mbbr$ belongs to the domain of the extended generator $\tilde{\mca}$ of a c\`{a}dl\`{a}g homogeneous Feller Markov process $Y$ taking values in $\mbbr$ if there exists a measurable function $g:\mbbr\to\mbbr$ such that the process
\begin{equation*}
f(Y_t)-\int_0^t g(Y_s)ds, \qquad t\in\mbbr^+,
\end{equation*}
is well defined and is a local martingale with respect to the natural filtration of $Y$ and every measure $P_x(\cdot):=P(\cdot|Y_0=x),\ x\in\mbbr$.
For such a pair $(f,g)$, we write $f\in\mbox{Dom}(\tilde{\mca})$ and $\tilde{\mca} f\overset{EPE}=g$.
\end{Def}

\begin{Rem}
In the previous definition, the terminology ``Feller" means that the corresponding transition semigroup $T_t$ is a mapping $C_b(\mbbr)$ into $C(\mbbr)$.
When it comes to $X$, its homogeneous, Feller and (strong) Markov properties are guaranteed by the argument in \cite[Theorem 6.4.6]{App09} and \cite[3.1.1 (ii)]{Mas07}.
\end{Rem} 

\begin{Rem}
When we consider the misspecified ergodic diffusion models, we also encounter the annoying integral term like \eqref{met}.
In that case, \cite{UchYos11} utilized the theory of the second order differential equations endowed with their infinitesimal generator (cf. \cite{ParVer01}) and It\^{o}'s formula to derive the asymptotic normality of the GQMLE. 
However, in our case, the same method cannot be applied since the infinitesimal generator of $X$ contains the integro-operator with respect to the L\'{e}vy measure of $Z$ and it is difficult to verify the existence and regularity of the corresponding equation.
\end{Rem}

Hereinafter $y^{(i)}$ is referred to as the $i$-th component of any vector $y$.
We consider the following EPEs:
\begin{align}
\label{u1}
\tilde{\mca} f_1^{(j_1)}(x)&\overset{EPE}=-\frac{\p_{\gam^{(j_1)}} c(x,\gam^\star)}{c^3(x,\gam^\star)}(c^2(x,\gam^\star)-C^2(x)),\\
\label{u2}
\tilde{\mca} f_2^{(j_2)}(x)&\overset{EPE}=-\frac{\p_{\al^{(j_2)}} a(x,\al^\star)}{c^2(x,\gam^\star)}(A(x)-a(x,\al^\star)),
\end{align}
for the extended generator $\tilde{\mca}$ of $X$, $j_1\in\{1,\dots,p_\gam\}$ and $j_2\in\{1,\dots,p_\al\}$.
The right-hand-side of each EPE corresponds to $g$ in \eqref{me}, and it is trivial that they identically 0 when the coefficients are correctly specified.

From now on, $E^x$ is referred to as the expectation operator with the initial condition $X_0=x$, that is,
\begin{equation*}
E^x[g(X_t)]=\int_\mbbr g(y)P_t(x,dy),
\end{equation*}
for any measurable function $g$.
The next proposition ensures the existence of the solutions of \eqref{u1} and \eqref{u2} and verifies their weighted H\"{o}lder continuity:
\begin{Prop}\label{EPE}
Under Assumption \ref{Moments}-\ref{Stability},
there exist unique solutions of \eqref{u1} and \eqref{u2}, and the solution vectors $f_1:=\left(f_1^{(j_1)}\right)_{j_1\in\{1,\dots,p_\gam\}}$ and $f_2:=\left(f_2^{(j_2)}\right)_{j_2\in\{1,\dots,p_\al\}}$ satisfy
\begin{equation*}
\sup_{x,y\in\mbbr, x\neq y}\frac{|f_i(x)-f_i(y)|}{|x-y|^{1/{p_i}}(1+|x|^{q_iK_i}+|y|^{q_iK_i})}<\infty, \quad \mbox{for} \ i\in\{1,2\}, 
\end{equation*}
where any $p_i\in(1,\infty)$, $q_i={p_i}/(p_i-1)$, and some positive constants $K_1$ and $K_2$.
Furthermore, 
\begin{equation*}
f_1(X_t)+\int_0^t \frac{\p_\gam c(X_s,\gam^\star)}{c^3(X_s,\gam^\star)}(c^2(X_s,\gam^\star)-C^2(X_s))ds,
\end{equation*}
and
\begin{equation*}
 f_2(X_t)+\int_0^t \frac{\p_\al a(X_s,\al^\star)}{c^2(X_s,\gam^\star)}(A(X_s)-a(X_s,\al^\star))ds
\end{equation*}
are $L_2$-martingale with respect to $(\mcf_t, P_x)$ for every $x\in\mbbr$, and their explicit forms are given as follows:
\begin{align*}
f_1(x)&=\int_0^\infty E^x\left[\frac{\p_\gam c(X_t,\gam^\star)}{c^3(X_t,\gam^\star)}(c^2(X_t,\gam^\star)-C^2(X_t))\right]dt,\\
f_2(x)&=\int_0^\infty E^x\left[\frac{\p_\al a(X_t,\al^\star)}{c^2(X_t,\gam^\star)}(A(X_t)-a(X_t,\al^\star))\right]dt.
\end{align*}
\end{Prop}

\begin{Rem} 
Thanks to the result of the previous theorem and assumptions on the coefficients, 
\begin{equation*}
f_1(X_t)+\int_0^t \frac{\p_\gam c(X_s,\gam^\star)}{c^3(X_s,\gam^\star)}(c^2(X_s,\gam^\star)-C^2(X_s))ds,
\end{equation*}
and 
\begin{equation*}
f_2(X_t)+\int_0^t \frac{\p_\al a(X_s,\al^\star)}{c^2(X_s,\gam^\star)}(A(X_s)-a(X_s,\al^\star))ds 
\end{equation*}
have finite second-order moments.
Thus, slightly refining the argument in \cite[the proof of Proposition V\hspace{-.1em}I\hspace{-.1em}I 1.6]{RevYor99} with the monotone convergence theorem, the $L_2$-martingale property of them with respect to $(\mcf_t, P_x)$ can be replaced by the  $L_2$-martingale property with respect to $(\mcf_t, P)$ in the previous proposition. 
\end{Rem}

Building on the previous proposition, now we can obtain the asymptotic normality of $\sqrt{T_n}(\tes-\theta^\star)$:
\begin{Thm}\label{AN}
Under Assumptions \ref{Moments}-\ref{nd},
there exists a nonnegative definite matrix $\Sig\in\mbbr^p\otimes\mbbr^p$ such that
\begin{equation*}
\sqrt{T_n}(\tes-\theta^\star)\overset{\mcl}\longrightarrow N(0,\Gam^{-1}\Sig(\Gam^{-1})^\top),
\end{equation*}
and the form of $\Sig:=\begin{pmatrix}\Sig_\gam&\Sig_{\al\gam}\\\Sig_{\al\gam}^\top&\Sig_{\al}\end{pmatrix}$ is given by:
\begin{align*}
&\Sig_\gam=4\int_\mbbr\int_\mbbr\left(\frac{\p_\gam c(x,\gam^\star)}{c^3(x,\gam^\star)}C^2(x)z^2+f_1(x+C(x)z)-f_1(x)\right)^{\otimes2}\pi_0(dx)\nu_0(dz),\\
&\Sig_{\al\gam}=-4\int_\mbbr\int_\mbbr\left(\frac{\p_\gam c(x,\gam^\star)}{c^3(x,\gam^\star)}C^2(x)z^2+f_1(x+C(x)z)-f_1(x)\right)\\
&\qquad \qquad \quad\left(\frac{\p_\al a(x,\al^\star)}{c^2(x,\gam^\star)}C(x)z+f_2(x+C(x)z)-f_2(x)\right)^\top\pi_0(dx)\nu_0(dz),\\
&\Sig_\al=4\int_\mbbr\int_\mbbr\left(\frac{\p_\al a(x,\al^\star)}{c^2(x,\gam^\star)}C(x)z+f_2(x+C(x)z)-f_2(x)\right)^{\otimes2}\pi_0(dx)\nu_0(dz).
\end{align*}
\end{Thm}

\begin{Rem}\label{partly}
If either of the coefficients is correctly specified, the right-hand side of the associated EPE \eqref{u1} or \eqref{u2} is identically 0.
Let $\gam_0$ and $\al_0$ be the elements of $\Theta_\gam$ and $\Theta_\al$ whose definitions are introduced in Rem \ref{optim}.
Then we have
\begin{align*}
&\Sig_\gam=4\int_\mbbr\left(\frac{\p_\gam c(x,\gam_0)}{c(x,\gam_0)}\right)^{\otimes2}\pi_0(dx)\int_\mbbr z^4\nu_0(dz),\\
&\Sig_{\al\gam}=-4\int_\mbbr\int_\mbbr\left(\frac{\p_\gam c(x,\gam_0)}{c(x,\gam_0)}z^2\right)\left(\frac{\p_\al a(x,\al^\star)}{c(x,\gam_0)}z+f_2(x+c(x,\gam_0)z)-f_2(x)\right)^\top\pi_0(dx)\nu_0(dz),\\
&\Sig_\al=4\int_\mbbr\int_\mbbr\left(\frac{\p_\al a(x,\al^\star)}{c(x,\gam_0)}z+f_2(x+c(x,\gam_0)z)-f_2(x)\right)^{\otimes2}\pi_0(dx)\nu_0(dz),
\end{align*}
in the case that the scale coefficient is correctly specified and
\begin{align*}
&\Sig_\gam=4\int_\mbbr\int_\mbbr\left(\frac{\p_\gam c(x,\gam^\star)}{c^3(x,\gam^\star)}C^2(x)z^2+f_1(x+C(x)z)-f_1(x)\right)^{\otimes2}\pi_0(dx)\nu_0(dz),\\
&\Sig_{\al\gam}=-4\int_\mbbr\int_\mbbr\left(\frac{\p_\gam c(x,\gam^\star)}{c^3(x,\gam^\star)}C^2(x)z^2+f_1(x+C(x)z)-f_1(x)\right)\left(\frac{\p_\al a(x,\al_0)}{c^2(x,\gam^\star)}C(x)z\right)^\top\pi_0(dx)\nu_0(dz),\\
&\Sig_\al=4\int_\mbbr\left(\frac{\p_\al a(x,\al_0)}{c^2(x,\gam^\star)}C(x)\right)^{\otimes2}\pi_0(dx),
\end{align*}
in the case that the drift coefficient is correctly specified.
\end{Rem}

\begin{Rem}\label{Mcon}
Let $Y$ be a random variable which obeys $N(0,\Gam^{-1}\Sig(\Gam^{-1})^\top)$.
As a consequence of Theorem \ref{TPE} and Theorem \ref{AN}, we have 
\begin{equation}\label{ME}
E\left[f\left(\sqrt{T_n}(\tes-\theta^\star)\right)\right]\to E[f(Y)],
\end{equation}
for any polynomial growth function $f$.
It can be shown in the following way:
For any $q>1$, it follows from \cite[Lemma 2.2.8]{Dur10} and Theorem \ref{TPE} that
\begin{align*}
E\left[\left|\sqrt{T_n}(\tes-\theta^\star)\right|^q\right]&=\int_0^\infty qx^{q-1} P\left(\left|\sqrt{T_n}(\tes-\theta^\star)\right|>x\right)dx\\
&\lesssim \int_0^1 x^{q-1}dx+\int_1^\infty x^{-q}dx<\infty.
\end{align*} 
Hence $\left|\sqrt{T_n}(\tes-\theta^\star)\right|^q$ is asymptotically uniformly integrable from Markov's inequality, and \cite[Theorem 2.20]{Van00} implies \eqref{ME}.
\end{Rem}

\begin{Rem}\label{conj}
In this remark, we suppose that the data-generating model defined on the probability space $(\Omega, \mcf, (\mcf_t)_{t\in\mbbr_+},P)$ is supposed to be
\begin{equation}\label{twMolde}
dY_t=A(Y_t)dt+B(Y_t)dW_t+C(Y_{t-})dZ_t,
\end{equation}
where $W$ is a standard Wiener process independent of $(Y_{0},Z)$, $\mcf_t:=\sig(Y_0)\vee\sig((W_s,Z_s);s\leq t)$ and $B:\mbbr\to\mbbr$ is a measurable function.
We look at the following parametric model:
\begin{equation*}
dY_t = a(Y_t,\al)dt+b(Y_t,\gam)dW_t+c(Y_{t-},\gam)dZ_t,
\end{equation*}
where $b:\mbbr\times\Theta_\gam\to\mbbr$ is a measurable function. 
Here other ingredients are similarly defined as above and we use the same notations for its transition probability, invariant measure, and so on.
When the true coefficients $(A,B,C)$ are correctly specified, the GQMLE still has asymptotic normality and the sufficient conditions for it are easy to check (cf. \cite{Mas13-2}).
However, we note that it is difficult to give such conditions when they are misspecified.
This is because our methodology using the martingale representation theorem becomes insufficient due to the presence of Wiener component in the deviation of the asymptotic variance (see, the proof of Theorem \ref{AN}).
To formally derive a similar result to Theorem \ref{AN}, we may additionally have to impose the following condition:
\medskip

\textbf{Condition A}: There exists a unique $C^2$-solution $f$ on $\mbbr$ of 
\begin{equation}\label{PE}
\mca f(x)=A(x)\p_x f(x)+\frac{1}{2}B(x)\p_x^2 f(x)+\int_\mbbr (f(x+C(x)z)-f(x)-\p_x f(x)C(x)z)\nu_0(dz)=g(A(x),B(x),C(x)),
\end{equation}
where $g(A(x),B(x),C(x))$ is a specific function satisfying 
\begin{equation*}
\int_\mbbr g(A(x),B(x),C(x))\pi_0(dx)=0.
\end{equation*}
Furthermore, the first and second derivatives of $f$ are of at most polynomial growth.

\medskip
Under \textbf{Condition A}, the limit distribution of the GQMLE can be derived by combining the proof of \cite{UchYos12} and Theorem \ref{AN}.
It is known that the theory of viscosity solutions for integro-differential equations ensures the existence of $f$ in limited situation, for instance, see \cite{BarBucPar97}, \cite{BarImb08}, \cite{Ham16} and \cite{HamMor16}.
However, it is not so for the regularity of $f$. 
As another attempt to confirm \textbf{Condition A}, the associated EPE $\tilde{\mca}\tilde{f}\overset{EPE}=g$ may possibly be helpful.
This is because the existence and uniqueness of the solution $\tilde{f}$ of the EPE can be verified in an analogous way to Theorem \ref{EPE}, and if $\tilde{f}$ admits $C^2$-property and growth conditions in \textbf{Condition A}, then $\tilde{f}$ satisfies \eqref{PE}.
The latter argument can formally be shown as follows:
\medskip

It is enough to check $\mca\tilde{f}=g$.
Since $\tilde{f}(Y_t)-\int_0^t g(A(Y_s),B(Y_s),C(Y_s))ds$ is a martingale with respect to $(\mcf_t,P_x)$ for all $x\in\mbbr$, we have 
\begin{equation*}
E^x\left[\tilde{f}(Y_t)-\int_0^t g(A(Y_s),B(Y_s),C(Y_s))ds\right]=\tilde{f}(x).
\end{equation*}
Hence it follows from It\^{o}'s formula that as $t\to0$,
\begin{align*}
\left|\frac{E^x[\tilde{f}(Y_t)]-\tilde{f}(x)}{t}-g(A(x),B(x),C(x))\right|&=\left|\frac{1}{t}\int_0^t \left(E^x[g(A(Y_s),B(Y_s),C(Y_s))]-g(A(x),B(x),C(x))\right)ds\right|\\
&=\left|\frac{1}{t}\int_0^t \int_0^sE^x[\mca g(A(Y_u),B(Y_u),C(Y_u))]duds\right|\lesssim t \to 0.
\end{align*}

\medskip
In this sketch, we implicitly assume suitable regularity and moment conditions on each ingredient, but they are reduced to be conditions on the true coefficients $(A,B,C)$.
Thus, verifying the behavior of 
\begin{equation*}
\tilde{f}(x)=\int_0^\infty E^x[g(A(Y_t),B(Y_t),C(Y_t))] dt=\int_0^\infty \int_\mbbr g(A(y),B(y),C(y))P_t(x,dy)
\end{equation*}
leads to \textbf{Condition A}.
Just for L\'{e}vy driven Ornstein-Uhlenbeck models, we can observe the property of $\tilde{f}(x)=\int_0^\infty E^x[g(A(Y_t),B(Y_t),C(Y_t))] dt$ based on the explicit form of the solution (cf. Example \ref{ex1}).
Although, for general L\'{e}vy driven SDEs, the gradient estimates of their transition probability making use of Malliavin calculus have been investigated lately (cf. \cite{Wan14}, \cite{WanXuZha15}, and the references therein), the property of $\tilde{f}(x)=\int_0^\infty E^x[g(A(Y_t),B(Y_t),C(Y_t))] dt$ is still difficult to be checked as far as the author knows.
Since these are out of range of this paper, we will not treat them later.
\end{Rem}

\begin{Example}\label{ex1}
Here we consider the following Ornstein-Uhlenbeck model:
\begin{equation*}
dX_t=-\al X_tdt+dZ_t,
\end{equation*}
for a L\'{e}vy process $Z$ not necessarily being pure-jump type and a positive constant $\al$.
Applying It\^{o}'s formula to $\exp(\al t)X_t$, we have 
\begin{equation*}
X_t=X_0\exp(-\al t)+\int_0^t \exp(\al(s-t))dZ_s
\end{equation*}
and 
\begin{equation*} 
E^x[f(X_t)]=\int_\mbbr f(x\exp(-\al t)+y)p_t(dy)
\end{equation*}
for a suitable function $f$. Here $p_t$ is the probability distribution function of $\int_0^t \exp(\al(s-t))dZ_s$ whose characteristic function $\hat{p}_t(\cdot)$ is given by: 
\begin{equation}\label{cf}
\hat{p}_t(u)=\exp\left\{\int_0^t \psi(\exp(\al(s-t))u)ds\right\},
\end{equation}
for $\psi(u):=\log E[\exp(iuZ_1)]$ (cf. \cite[Theorem 3.1]{SatYam84}).
In this case, $X$ fulfills Assumption \ref{Stability} provided that Assumption \ref{Moments}-(1) holds, and that the L\'{e}vy measure $\nu_0$ of $Z$ has a continuously differentiable positive density $g$ on an open neighborhood around the origin (for more details, see \cite[Section 5]{Mas13-2}).
Then, the characteristic function $\hat{p}(\cdot)$ of the invariant measure $\pi_0$ is given by
\begin{equation*}
\hat{p}(u)=\exp\left\{\int_0^\infty \psi(\exp(-\al s)u)ds\right\}.
\end{equation*}
Under such condition, if $f$ is differentiable and itself and its derivative are of at most polynomial growth, we have 
\begin{align*}
\left|\p_x \left(\int_0^\infty E^x[f(X_t)]dt\right) \right|&=\left|\int_0^\infty \left(\int_\mbbr \p_x f(x\exp(-\al t)+y)p_t(dy)\right)\exp(-\al t)dt\right|\\
&\lesssim \int_0^\infty \left\{1+|x|^K+(1+|x|^{2K})\exp(-at)\right\}\exp(-\al t)dt\lesssim 1+|x|^{2K},
\end{align*}
for a positive constant K.
We can derive similar estimates with respect to its higher-order derivatives in the same way.
\end{Example}

Let $J$ be a L\'{e}vy process such that its moments of any-order exists and its triplet is $(0,b,\nu^J)$ (cf. \cite{App09}).
Here $b$ is allowed to be $0$.
Mimicking the previous example, we write $p^J_t$ as the probability distribution function of $ \int_0^t \exp(\al(s-t)) dJ_s$ for a positive constant $\al>0$ and $\psi^J(u)$ stands for $\log E[\exp(iuJ_1)]$ below.
Combining the argument in Remark \ref{conj} and Example \ref{ex1}, we obtain the following corollary:

\begin{Cor}
For a natural number $k\geq 2$, let $f$ be a polynomial growth $C^k$-function whose derivatives are of at most polynomial growth.
Suppose that the integral of $f$ with respect to the Borel probability measure $\pi_0$ whose characteristic function is $\exp\left\{\int_0^\infty \psi^J(\exp(-\al s)u)ds\right\}$ is $0$, and that $\nu^J$ has a continuously differentiable positive density on an open neighborhood around the origin.
Then, the function 
\begin{equation*}
g(x):=\int_0^\infty E^x\left[f\left(x\exp(-\al t)+\int_0^t \exp(\al(s-t))dJ_s\right)\right]dt=\int_0^\infty\int_\mbbr f(x\exp(-\al t)+y)p^J_t(dy)dt
\end{equation*}
on $\mbbr$ is the unique solution of the following (first or second order) integro-differential equation
\begin{equation}\label{ide}
-\al x\p_x g(x)+\frac{1}{2}b\p_x^2 g(x)+ \int_{\mbbr } \left(g(x+z)-g(x)-\p g(x)z\right)\nu^J(dz)=f(x),
\end{equation}
and moreover, $g$ is also a polynomial growth $C^k$-function.
\end{Cor}

\begin{Rem}
If the L\'{e}vy measure $\nu^J$ is symmetric (i.e. the imaginary part of $\psi^J$ is 0), the equation \eqref{ide} is solvable for many odd functions $f$ as a matter of course.
More specifically, for $k\in\mbbn$ and $f(x)=x^{2k+1}$, the solution $g$ is 
\begin{align*}
g(x)&=\int_0^\infty\int_\mbbr (x\exp(-\al t)+y)^{2k+1} p^J_t(dy)dt\\
&=\int_0^\infty \int_\mbbr \sum_{i=0}^{2k+1} \frac{(2k+1)!}{i! (2k+1-i)!} (x\exp(-\al t))^i y^{2k+1-i} p^J_t(dy)dt.
\end{align*}
By observing the derivatives of the characteristic function, $\int_\mbbr y^{2k+1-i} p^J_t(dy)$ can be expressed by the moments of $J$, hence the explicit expression of $g$ is available.
\end{Rem}
  
\begin{Rem}
Beside the estimation of $\theta$, what is of special interest is the inference for $\nu_0$ which may often be an infinite dimensional parameter.
Even for $(A,C)$ being constant and specified (i.e. $X$ is a L\'{e}vy process with drift), it may be interest in its own right and enormous papers have addressed this problem so far.
We refer to \cite{Mas15} for comprehensive accounts under $Z$ being assumed to have a certain parametric structure.
As for the situation where just a few information on $Z$ is available, one of plausible attempts is the method of moments proposed in \cite{Fig08}, \cite{Fig09}, and \cite{NicReiSohTra16}, for example.
Especially \cite{NicReiSohTra16} established a Donsker-type functional limit theorem for empirical processes arising from high-frequently observed L\'{e}vy processes.
When the coefficients $A$ and $C$ are nonlinear functions but specified, the residual based method of moments for $\nu_0$ by \cite{MasUeh17-1} is effective: using the GQMLE $\tes:=(\ges,\aes)$, we have 
\begin{align*}
&\frac{1}{T_n}\sumj \vp\left(\frac{\D_j X-h_na_{j-1}(\aes)}{c_{j-1}(\ges)}\right)\overset{P}\to \int_\mbbr \vp(z)\nu_0(dz),\\
&\hat{D}_n\sqrt{T_n}\begin{pmatrix}\tes-\theta_0\\ \frac{1}{T_n}\sumj \vp\left(\frac{\D_j X-h_na_{j-1}(\aes)}{c_{j-1}(\ges)}\right)- \int \vp(z)\nu_0(dz)\end{pmatrix} \overset{\mcl}\rightarrow N(0,I_{p+q}),
\end{align*}
for an appropriate $\mbbr^q$-valued function $\vp$ and a $(p+q)\times(p+q)$ matrix $\hat{D}_n$ which can be constructed only by the observations. 
For instance, we can choose $\vp(z)=z^r$ and $\vp(z)=\exp(iuz)-1-iuz$ (to estimate the $r$-th cumulant of $Z$ and the cumulant function of $Z$, respectively) as $\vp$; see \cite[Assumption 2.7]{MasUeh17-1} for the precise conditions on $\vp$.
As for misspecified case, if the misspecification is confined within the drift coefficient, then this scheme is still valid thanks to the faster diminishment of the mean activity in small time (cf. Remark \ref{decay}).
\end{Rem}

\section{Numerical experiments}\label{NE}
\begin{figure}[t]
\caption{The plot of the density functions of (i) $NIG(10,0,10,0)$ (black dotted line), (ii) $bGamma(1,\sqrt{2},1,\sqrt{2})$ (green line), (iii) $NIG(25/3,20/3,9/5,-12/5)$ (blue line), and $N(0,1)$ (red line).}
\begin{center}
   \includegraphics[width=80mm]{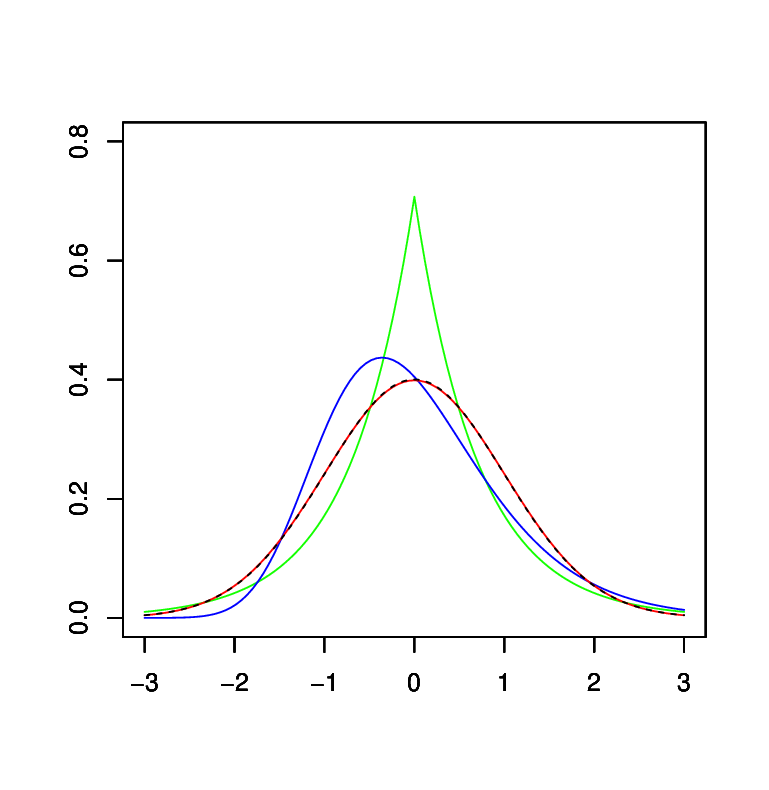}
  \end{center}
\label{compdens}
\end{figure}

\begin{figure}[t]
\caption{The boxplot of case (i); the target optimal values are described by dotted lines.}
\begin{center}
   \includegraphics[width=150mm]{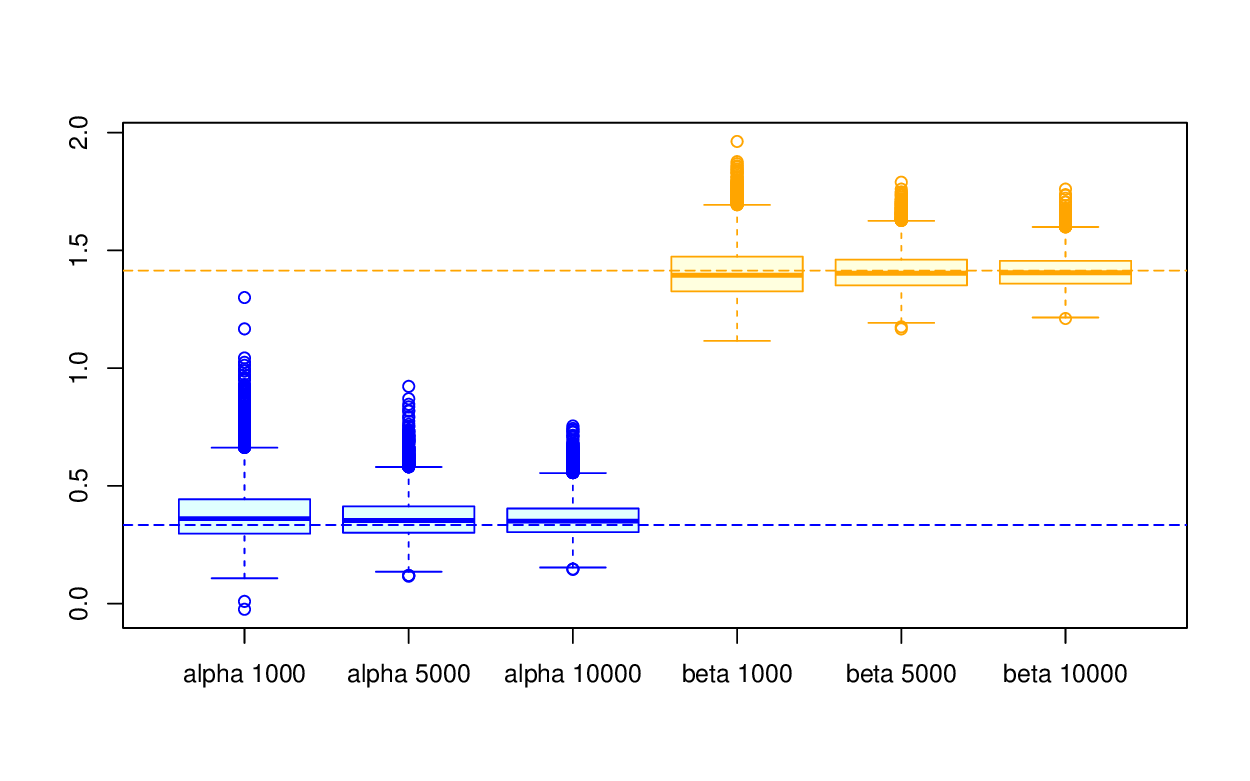}
  \end{center}
\end{figure}

\begin{figure}[t]
\caption{The boxplot of case (ii); the target optimal values are described by dotted lines.}
\begin{center}
   \includegraphics[width=150mm]{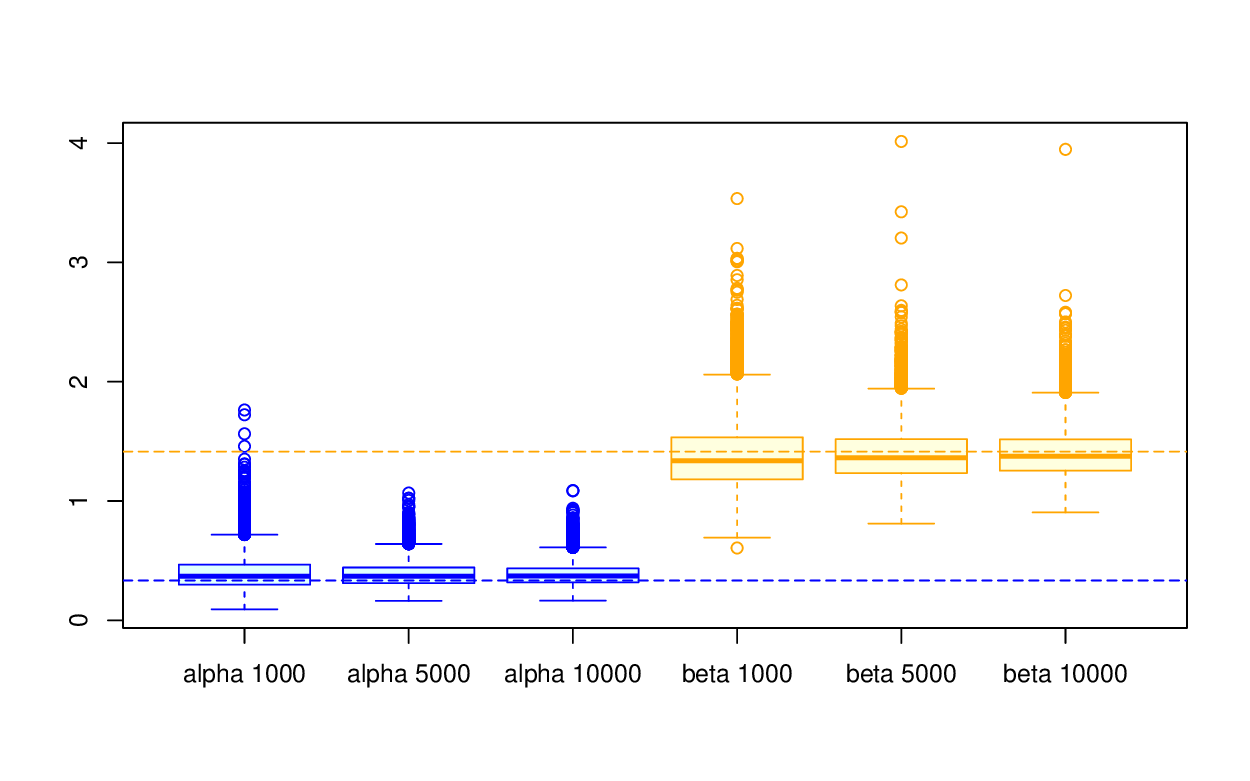}
  \end{center}
\end{figure}

\begin{figure}[t]
\caption{The boxplot of case (iii); the target optimal values are described by dotted lines.}
\begin{center}
   \includegraphics[width=150mm]{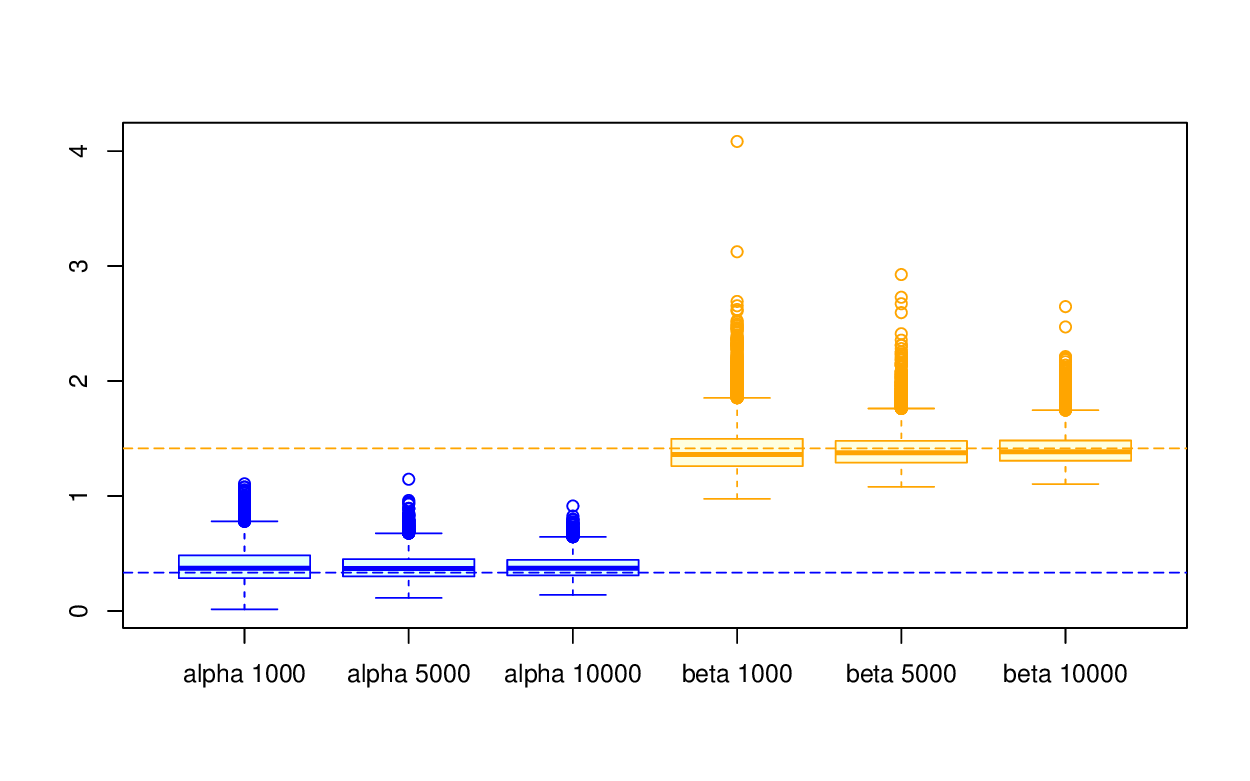}
  \end{center}
\end{figure}

\begin{table}[t]
\begin{center}
\caption{The performance of our estimators; the mean is given with the standard deviation in parenthesis. The target optimal values are given in the first line of each items.}
\scalebox{0.85}{\begin{tabular}{ccccccccccc}
\hline
&&&&&&&&&&\\[-3.5mm]
$T_n$ & $n$ & $h_n$ & \multicolumn{2}{c}{(i) (0.33,1.41)} &  \multicolumn{2}{c}{(ii) (0.37, 1.41)} & \multicolumn{2}{c}{(iii) (0.37, 1.41)} &\multicolumn{2}{c}{diffusion (0.33, 1.41)} \\ 
&&& $\aes$ & $\ges$ & $\aes$ & $\ges$ & $\aes$ & $\ges$ & $\aes$ & $\ges$ \\ \hline
50& 1000 & 0.05&0.38&1.41&0.40&1.39&0.40&1.39&0.38&1.41 \\
&&&(0.12)&(0.11)&(0.16)&(0.29)&(0.15)&(0.19)&(0.13)&(0.10)\\ 
100& 5000 & 0.02&0.37&1.41&0.39&1.39&0.38&1.39&0.36&1.41\\
&&&(0.09)&(0.08)&(0.11)&(0.23)&(0.11)&(0.15)&(0.09)&(0.08)\\ 
100&10000&0.01&0.36&1.41&0.37&1.39&0.38&1.40&0.36&1.41\\  
&&&(0.08)&(0.07)&(0.09)&(0.22)&(0.10)&(0.15)&(0.08)&(0.07)\\ \hline
\end{tabular}
\label{Result}}
\end{center}
\end{table}
We suppose that the data-generating model is the following L\'{e}vy driven Ornstein-Uhlenbeck model:
\begin{equation*}
dX_t=-\frac{1}{2}X_tdt+dZ_t,\quad X_0=0,
\end{equation*}
and that the parametric model is described as:
\begin{equation*}
dX_t=\alpha (1-X_t)dt+\frac{\gam}{\sqrt{1+X_t^2}}dZ_t, \quad \al,\gam>0.
\end{equation*}
The functional form of the coefficients is the same in \cite[Example 3.1]{UchYos11}.
We conduct numerical experiments in three situations: (i) $\mcl(Z_t)=NIG(10,0,10t,0)$, (ii) $\mcl(Z_t)=bGamma(t,\sqrt{2},t,\sqrt{2})$, and (iii) $\mcl(Z_t)=NIG(25/3,20/3,9/5t,-12/5t)$.
$NIG$ (normal inverse Gaussian) random variable is defined by the normal mean-variance mixture of inverse Gaussian random variable, and $bGamma$ (bilateral Gamma) random variable is defined by the difference of two independent Gamma random variables. 
For their technical accounts, we refer to \cite{Bar98} and \cite{KucTap08-1}.
To visually observe their non-Gaussianity, each density function at $t=1$ is plotted with the density of $N(0,1)$ in Figure \ref{compdens} altogether.
By taking the limit of \eqref{cf}, the characteristic function $\hat{p}(\cdot)$ of the invariant measure $\pi_0$ is given by
\begin{equation}
\hat{p}(u)=\exp\left\{\int_0^\infty \psi\left(\exp\left(-\frac{s}{2}\right)u\right)ds\right\},
\end{equation}
where $\psi(u):=\log E[\exp(iuZ_1)]$.
Differentiating $\hat{p}$, we have $\tilde{\kappa}_j=2\kappa_j/j$ for the $j$-th cumulant $\tilde{\kappa}_j$ (resp. $\kappa_j$) of $Y\sim \pi_0$ (resp. $Z_1$).
Hence we obtain
\begin{align*}
&\mbbg_{1}(\gam)=-2\log\gam-\frac{2}{\gam^2}+\int_\mbbr \log(1+x^2)\pi_0(dx),\\
&\mbbg_{2}(\al)=-\frac{1}{\gam^\star}\left\{\frac{1}{4}\int_\mbbr x^3\pi_0(dx)+\al\left(1-\int_\mbbr x^3\pi_0(dx)+\int_\mbbr x^4\pi_0(dx)\right)+\al^2\left(3-2\int_\mbbr x^3\pi_0(dx)+\int_\mbbr x^4\pi_0(dx)\right)\right\}.
\end{align*} 
By solving the estimating equations, the target optimal values are given by
\begin{align*}
\gam^\star=\sqrt{2}, \ \al^\star=\frac{1-\int_\mbbr x^3\pi_0(dx)+\int_\mbbr x^4\pi_0(dx)}{2(3-2\int_\mbbr x^3\pi_0(dx)+\int_\mbbr x^4\pi_0(dx))}.
\end{align*}
In the calculation, we used $\int_\mbbr x\pi_0(dx)=0$ and $\int_\mbbr x^2\pi_0(dx)=1$.
Thus, in each case, the optimal parameter $\theta^\star:=(\al^\star,\gam^\star)$ is given as follows: (i) $\theta^\star=\left( 803/2406,\sqrt{2}\right)\approx(0.3337, 1.4142)$, (ii) $\theta^\star=\left(11/30,\sqrt{2}\right)\approx(0.3667, 1.4142)$, and (iii) $\theta^\star=\left( 609/1658, \sqrt{2}\right)\approx(0.3673, 1.4142)$ (we write approximated values  obtained by rounding off $\theta^\star$ to four decimal places).
Solving the corresponding estimating equations, our staged GQMLE are calculated  as:
\begin{align*}
\hat{\al}_n=-\frac{\sum_{j=1}^n (X_{j-1}-1)(X_j-X_{j-1})(X_{j-1}^2+1)}{h_n\sum_{j=1}^n(X_j-1)^2(X_{j-1}^2+1)}, \ \ges=\sqrt{\frac{1}{nh_n}\sum_{j=1}^n(X_j-X_{j-1})^2(X_{j-1}^2+1)}.
\end{align*}
We generated 10000 paths of each SDE based on Euler-Maruyama scheme and constructed the estimators along with the above expressions, independently.
In generating the small time increments of the driving noises, we used the function \texttt{rng} equipped to YUIMA package in R \cite{BroFuk14}.
Together with the diffusion case $\left(\theta^\star=\left(1/3,\sqrt{2}\right)\approx(0.3333, 1.4142)\right)$, the mean and standard deviation of each estimator is shown in Table \ref{Result} where $n$ and $h_n=5n^{-2/3}$ denote the sample size and observation interval, respectively.
We also present their boxplots to enhance the visibility.
We can observe the followings from the table and boxplots:

\begin{itemize}
\item Overall, the estimation accuracy of $\tes$ improves as $T_n$ and $n$ increase and $h_n$ decrease, and this tendency reflects our main result.
\item The result of case (i) is almost the same as the diffusion case. This is thought to be based on the well-known fact that $NIG(\del,0,\del t,0)$ tends to $N(0,t)$ in total variation norm as $\del\to\infty$ for any $t>0$.
Indeed, Figure \ref{compdens} shows that the density functions of $NIG(10,0,10,0)$ and $N(0,1)$ are virtually the same.
\item Concerning case (ii), the standard deviation of $\ges$ is relatively worse than the other cases. This is natural because the asymptotic variance of $\ges$ includes the forth-order-moment of $Z$, and $bGamma(1,\sqrt{2},1,\sqrt{2})$ has the highest kurtosis value as can be seen from Figure \ref{compdens}. 
\item In case (iii), the performance of $\aes$ is the worst in this experiment. This may cause from the fact that only $NIG(25/3,20/3,9/5,-12/5)$ is not symmetric.
\end{itemize}

\section{Appendix}\label{AP}
%
%
Throughout the proofs, for functions $f$ on $\mbbr\times\Theta$, we will sometimes write $f_s$ and $f_{j-1}$ instead of $f_s(\theta^\star)$ and $f_{j-1}(\theta^\star)$ just for simplicity.
 \medskip

\noindent\textbf{Proof of Theorem \ref{TPE}}
\indent In light of our situation, it is sufficient to check the conditions [A1''], [A4'] and [A6] in \cite{Yos11} for $\mbbg_{1,n}$ and $\mbbg_{2,n}$, respectively. 
For the sake of convenience, we simply write $\mbby_{1,n}(\gam):=\mbbg_{1,n}(\gam)-\mbbg_{1,n}(\gam^\star)$ and $\mbby_{2,n}(\al):=\mbbg_{2,n}(\al)-\mbbg_{2,n}(\al^\star)$ below.
Without loss of generality, we can assume $p_\gam=p_\al=1$.
First we treat $\mbbg_{1,n}(\cdot)$.
The conditions hold if we show
\begin{align}
&\sup_{n\in\mbbn} E\left[|\sqrt{T_n}\p_\gam \mbbg_{1,n}(\gam^\star)|^K\right]<\infty, \label{mb1}\\
&\sup_{n\in\mbbn} E\left[|\sqrt{T_n}(\p_\gam^2\mbbg_{1,n}(\gam^\star)-\Gam_\gam)|^K\right]<\infty, \label{mb2}\\
&\sup_{n\in\mbbn} E\left[\sup_{\gam\in\Theta_\gam}|\p_\gam^3 \mbbg_{1,n}(\gam)|^K\right]<\infty,\label{mb3}\\
&\sup_{n\in\mbbn} E\left[\sup_{\gam\in\Theta_\gam}|\sqrt{T_n}\left(\mbby_{1,n}(\gam)-\mbby_1(\gam)\right)|^K\right]<\infty\label{mb4},
\end{align}
for any $K>0$.
The first two derivatives of $\mbbg_{1,n}$ are given by
\begin{align*}
&\p_\gam \mbbg_{1,n}(\gam)=-\frac{2}{T_n}\sum_{j=1}^n\left\{\frac{\p_\gam c_{j-1}(\gam)}{c_{j-1}(\gam)}h_n-\frac{\p_\gam c_{j-1}(\gam)}{c^3_{j-1}(\gam)}(\D_j X)^2\right\},\\
&\p_\gam^2 \mbbg_{1,n}(\gam)=-\frac{2}{T_n}\sumj\left\{\frac{\p_\gam^{2}c_{j-1}(\gam)c_{j-1}(\gam)-(\p_\gam c_{j-1})^{2}}{c^2_{j-1}(\gam)}h_n-\frac{\p_\gam^{2}c_{j-1}(\gam)c_{j-1}(\gam)-3(\p_\gam c_{j-1}(\gam))^{2}}{c^4_{j-1}(\gam)}(\D_j X)^2\right\}.
\end{align*}
We further decompose $\p_\gam \mbbg_{1,n}(\gam^\star)$ as
\begin{align*}
\p_\gam \mbbg_{1,n}(\gam^\star)=-\frac{2}{n}\sum_{j=1}^n\frac{\p_\gam c_{j-1}}{c^3_{j-1}}\left(c^2_{j-1}-C^2_{j-1}\right)+\frac{2}{T_n}\sum_{j=1}^n\frac{\p_\gam c_{j-1}}{c^3_{j-1}}\left\{(\D_j X)^2-h_nC^2_{j-1}\right\}.
\end{align*}
Since the optimal parameter $\theta^\star$ is in $\Theta$, the interchange of the derivative and the integral implies that the function $\p_\gam c(x,\gam^\star)(c^2(x,\gam^\star)-C^2(x))/c^3(x,\gam^\star)$ is centered in the sense that its integral with respect to $\pi_0$ is 0.
Thus \cite[Lemma 4.3]{Mas13-1} and \cite[Lemma 5.3]{MasUeh17-1} lead to \eqref{mb1} and \eqref{mb4}.
We also have
\begin{align*}
\p_\gam^2 \mbbg_{1,n}(\gam)&=-\frac{2}{n}\sumj\left\{\frac{\p_\gam^{2}c_{j-1}(\gam)c_{j-1}(\gam)-(\p_\gam c_{j-1})^{2}}{c^2_{j-1}(\gam)}-\frac{\p_\gam^{2}c_{j-1}(\gam)c_{j-1}(\gam)-3(\p_\gam c_{j-1}(\gam))^{2}}{c^4_{j-1}(\gam)}C^2_{j-1}\right\}\\
&+\frac{2}{T_n}\sumj\frac{\p_\gam^{2}c_{j-1}(\gam)c_{j-1}(\gam)-3(\p_\gam c_{j-1}(\gam))^{2}}{c^4_{j-1}(\gam)}\left\{(\D_j X)^2-h_nC^2_{j-1}\right\}.
\end{align*}
Again applying \cite[Lemma 4.3]{Mas13-1} and \cite[Lemma 5.3]{MasUeh17-1}, we obtain \eqref{mb2}.
Via simple calculation, the third and fourth-order derivatives of $\mbbg_{1,n}$ can be represented as
\begin{equation*}
\p_\gam^i\mbbg_{1,n}(\gam)=\frac{1}{n}\sumj g_{j-1}^i(\gam)+\frac{1}{T_n}\sumj \tilde{g}_{j-1}^i(\gam) \left\{(\D_j X)^2-h_n C^2_{j-1}\right\}, \quad \mbox{for} \ i\in\{3,4\},
\end{equation*} 
with the matrix-valued functions $g^i(\cdot,\cdot)$ and $\tilde{g}^i(\cdot,\cdot)$ defined on $\mbbr\times\Theta_\gam$, and these are of at polynomial growth with respect to $x\in\mbbr$ uniformly in $\gam$.
Hence \eqref{mb3} follows from Sobolev's inequality (cf. \cite[Theorem 1.4.2]{Ada73}).
Thus \cite[Theorem 3-(c)]{Yos11} leads to the tail probability estimates of $\ges$.
From Taylor's expansion, we get
\begin{align*}
&\mbby_{2,n}(\al)\\
&=\frac{1}{T_n}\sumj \frac{2\D_j X(a_{j-1}(\al)-a_{j-1}(\al^\star))+h_n(a_{j-1}^2(\al^\star)-a_{j-1}^2(\al))}{c^2_{j-1}(\gam^\star)}\\
&+\left(\int_0^1\frac{1}{(T_n)^{3/2}}\sumj \left\{2\D_j X(a_{j-1}(\al)-a_{j-1}(\al^\star))+h_n(a_{j-1}^2(\al^\star)-a_{j-1}^2(\al))\right\}\p_\gam c^{-2}_{j-1}(\gam^\star+u(\ges-\gam^\star))du\right)\\
&\qquad(\sqrt{T_n}(\ges-\gam^\star))\\
&:= \tilde{\mbby}_{2,n}(\al)+\bar{\mbby}_{2,n}(\al)(\sqrt{T_n}(\ges-\gam^\star)).
\end{align*}
Sobolev's inequality leads to
\begin{align*}
&E\left[\left|\sqrt{T_n}\bar{\mbby}_{2,n}(\al)\right|^K\right]\\
&\leq E\left[\sup_{\gam\in\Theta_\gam}\left|\frac{1}{T_n}\sumj \left\{2\D_j X(a_{j-1}(\al)-a_{j-1}(\al^\star))+h_n(a_{j-1}^2(\al^\star)-a_{j-1}^2(\al))\right\}\p_\gam c^{-2}_{j-1}(\gam)\right|^K\right]\\
&\lesssim \sup_{\gam\in\Theta_\gam}\left\{E\left[\left|\frac{1}{T_n}\sumj \left\{2\D_j X(a_{j-1}(\al)-a_{j-1}(\al^\star))+h_n(a_{j-1}^2(\al^\star)-a_{j-1}^2(\al))\right\}\p_\gam c^{-2}_{j-1}(\gam)\right|^K\right]\right.\\
&\left.\quad\qquad+E\left[\left|\frac{1}{T_n}\sumj \left\{2\D_j X(a_{j-1}(\al)-a_{j-1}(\al^\star))+h_n(a_{j-1}^2(\al^\star)-a_{j-1}^2(\al))\right\}\p_\gam^2 c^{-2}_{j-1}(\gam)\right|^K\right]\right\},
\end{align*}
for $K>1$.
The last two terms of the right-hand-side are finite from \cite[Lemma 5.3]{MasUeh17-1}, and the moment bounds of the three functions $\sqrt{T_n}\p_\al^i\bar{\mbby}_{2,n}(\al)$ ($i\in\{1,2,3\}$) can analogously be obtained.
Thus combined with the tail probability estimates of $\ges$ and Schwartz's inequality, it suffices to show the conditions for 
\begin{align*}
&\tilde{\mbbg}_{2,n}(\al):=-\frac{1}{T_n}\sumj \frac{(\D_j X-h_na_{j-1}(\al))^2}{h_nc^2_{j-1}(\gam^\star)},\\
&\tilde{\mbby}_{2,n}(\al):= \frac{1}{T_n}\sumj \frac{2\D_j X(a_{j-1}(\al)-a_{j-1}(\al^\star))+h_n(a_{j-1}^2(\al^\star)-a_{j-1}^2(\al))}{c^2_{j-1}(\gam^\star)},
\end{align*}
instead of $\mbbg_{2,n}(\al)$ and $\mbby_{2,n}(\al)$, respectively.
Since their estimates can be proved in a similar way to the first half, we omit the details.  $\square$

\medskip

To derive Proposition \ref{EPE}, we prepare the next lemma.
For $L_1$ metric $d(\cdot,\cdot)$ on $\mbbr$, we define the coupling distance $W(\cdot,\cdot)$ between any two probability measures $P$ and $Q$ by
\begin{equation*}
W(P,Q):=\inf\left\{\int_{\mbbr^2} d(x,y)d\mu(x,y): \mu\in M(P,Q)\right\}=\inf\left\{\int_{\mbbr^2} |x-y|d\mu(x,y): \mu\in M(P,Q)\right\},
\end{equation*}
where $M(P,Q)$ denotes the set of all probability measures on $\mbbr^2$ with marginals $P$ and $Q$.
$W(\cdot,\cdot)$ is called the probabilistic Kantrovich-Rubinstein metric (or the first Wasserstein metric).
The following assertion gives the exponential estimates of $W(P_t(\cdot,\cdot),\pi_0)$:

\begin{Lem}
If Assumption \ref{Stability} holds, then for any $q>1$, there exists a positive constant $C_q$ such that for all $x\in\mbbr$,
\begin{equation*}
W(P_t(x,\cdot),\pi_0)\leq C_q\exp(-at)(1+|x|^q).
\end{equation*} 
\end{Lem}

\begin{proof}
We introduce the following Lipschitz semi-norm for a suitable real-valued function $f$ on $\mbbr$: 
\begin{equation*}
||f||_L:=\sup\{|f(x)-f(y)|/|x-y|: x\neq y \ \mbox{in} \ \mbbr \}.
\end{equation*}
From Kantorovich-Rubinstein theorem (cf. \cite[Theorem 11.8.2]{Dud02}) and Assumption \ref{Stability}, it follows that for all $x\in\mbbr$,
\begin{align*}
W(P_t(x,\cdot),\pi_0)&=\sup\left\{\left|\int_\mbbr f(y) \{P_t(x,dy)-\pi_0(dy)\}\right|: ||f||_L\leq1\right\}\\
&=\sup\left\{\left|\int_\mbbr (f(y)-f(0)) \{P_t(x,dy)-\pi_0(dy)\}\right|: ||f||_L\leq1\right\}\\
&\leq \sup\left\{\left|\int_\mbbr h(y) \{P_t(x,dy)-\pi_0(dy)\}\right|: |h(y)|\leq 1+|y|^q\right\}\\
&\leq C_q\exp(-at)(1+|x|^q).
\end{align*}
\end{proof}

\medskip
\noindent\textbf{Proof of Proposition \ref{EPE}}
\indent It is enough to check the conditions of \cite[Theorem 3.1.1 and Theorem 3.1.3]{VerKul11} for $p_\gam=p_\al=1$.
As was mentioned in the proof of Theorem \ref{TPE}, 
\begin{equation*}
g_1(x):=-\p_\gam c(x,\gam^\star)(c^2(x,\gam^\star)-C^2(x))/c^3(x,\gam^\star),
\end{equation*}
and 
\begin{equation*}
g_2(x):=-\p_\al a(x,\al^\star)(A(x)-a(x,\al^\star))/c^2(x,\gam^\star)
\end{equation*}
are centered.
In the following, we give the proof concerning $g_1$ and omit its index $1$ for simplicity.
The regularity conditions on the coefficients imply that there exist positive constants $L$ and $D$ such that 
\begin{equation*}
|g(x)-g(y)|\leq D (2+|x|^{L}+|y|^{L}) |x-y|.
\end{equation*}
Making use of the trivial inequalities $|x-y|^{l}\leq|x|^{l}+|y|^{l}$ and $|x|^l\leq 1\vee |x|^{L+l}$ for any $L>0$, $l\in(0,1)$ and $x,y\in\mbbr$, we have
\begin{equation*}
 \sup_{x,y\in\mbbr, x\neq y}\frac{|g(x)-g(y)|}{(2+|x|^{L+1-1/p}+|y|^{L+1-1/p}) |x-y|^{1/p}}<\infty,
\end{equation*}
for any $p>1$.
Recall that we put $h_L(x)=1+|x|^L$ in Assumption \ref{Stability}.
The inequality \eqref{Ergodicity} gives
\begin{equation*}
\int_\mbbr h_L(y) P_t(x,dy) \leq ||P_t(x,\cdot)-\pi_0(\cdot)||_{h_L}+\int_\mbbr (1+|y|^L) \pi_0(dy)\leq \left(C_L+\int_\mbbr (1+|y|^L) \pi_0(dy)\right) h_L(x).
\end{equation*}
We write $L'=L+1-1/p$ for abbreviation.
Building on this estimate and the previous lemma, the conditions of \cite[Theorem 3.1.1 and Theorem 3.1.3]{VerKul11} are satisfied with 
\begin{align*}
&p=p, \ q=\frac{p}{p-1}, \ d(x,y)=|x-y|, \ r(t)=\exp(-at), \ \phi(x)=1+|x|^{L'},\\
&\psi(x)=2^{q-1}\left(C_{qL'}+\int_\mbbr h_{qL'}(y)\pi_0(dy)\right) h_{qL'}(x),\\
&\chi(x)=2^{q^2-1}\left(C_{qL'}+\int_\mbbr h_{qL'}(y) \pi_0(dy)\right)^q\left(C_{q^2L'}+\int_\mbbr h_{q^2L'}(y)\pi_0(dy)\right)h_{q^2L'}(x),
\end{align*}
and here these symbols correspond to the ones used in \cite{VerKul11}.
As for $g_2$, the conditions can be checked as well.
Hence the desired result follows. $\square$

%
%

\medskip
To derive the asymptotic normality of $\tes$, the following CLT-type theorem for stochastic integrals with respect to Poisson random measures will come into the picture:
\begin{Lem}\label{CLTPRM}
Let $N(ds,dz)$ be a Poisson random measure associated with one-dimensional L\'{e}vy process defined on a stochastic basis $(\Omega, \mcf,(\mcf_t)_{t>0},P)$ whose L\'{e}vy measure is written as $\nu_0$.
Assume that a continuous vector-valued function $f$ on $\mbbr_+\times\mbbr\times\mbbr$ and a $\mcf_t$-predictable process $H_t$ satisfy:
\begin{enumerate}
\item For all $T>0$ and $k=2,4$, 
\begin{equation*}
E\left[\int_0^T\int_\mbbr |f(T,H_s,z)|^k\nu_0(dz)ds\right]<\infty,
\end{equation*}
and their exists a positive definite matrix $\Sig$ such that 
\begin{equation*}
E\left[\int_0^T\int_\mbbr f(T,H_s,z)^{\otimes2}\nu_0(dz)ds\right]\to \Sig,
\end{equation*}
as $T\to\infty$;
\item there exists $\delta>0$ such that 
\begin{equation*}
E\left[\int_0^T\int_\mbbr |f(T,H_s,z)|^{2+\delta}\nu_0(dz)ds\right]\to0,
\end{equation*}
as $T\to\infty$.
\end{enumerate}
Then, for the associated compensated Poisson random measure $\tilde{N}(ds,dz)$, we have
\begin{equation*}
\int_0^T\int_\mbbr f(T,H_s,z)\tilde{N}(ds,dz)\cil N(0,\Sig),
\end{equation*}
as $T\to\infty$.
 \end{Lem}

\begin{proof}
By Cramer-Wold device, it is sufficient to show only one-dimensional case.
This proof is almost the same as \cite[Theorem 14. 5. I]{DalVer08}.
For notational brevity, we set
\begin{equation*}
X_1(t):=\int_0^{t}\int_\mbbr f(T,H_s,z)\tilde{N}(ds,dz), \quad X_2(t):=\int_0^{t}\int_\mbbr |f(T,H_s,z)|^2\nu_0(dz)ds,
\end{equation*}
Introduce a stopping time $S:=\inf\{t>0: X_2(t)\geq \Sig\}$.
Note that $X_2(S)=\Sig$ because $X_2(t)$ is continuous.
Define a random function $\zeta(u,t)$ by
\begin{equation*}
\zeta(u,t)=\exp\left\{iuX_1(t\wedge S)+\frac{u^2}{2}X_2(t\wedge S)\right\}.
\end{equation*}
Applying It\^{o}'s formula, we obtain
\begin{align*}
\zeta(u,T)&=1+iu\int_0^{T\wedge S}\zeta(u,s-)dX_1(s)+\frac{u^2}{2}\int_0^{T\wedge S}\zeta(u,s-)dX_2(s)\\
&+\sum_{0<s\leq T\wedge S}(\zeta(u,s-)\exp\left\{iu\D X_1(s)\right\}-\zeta(u,s-)-iu\zeta(u,s-)\D X_1(s))\\
&=1+\int_0^{T\wedge S}\int_\mbbr\zeta(u,s-)\left(\exp\left\{iuf(T,H_s,z)\right\}-1\right)\tilde{N}(ds,dz)\\
&+\int_0^{T\wedge S}\int_\mbbr\zeta(u,s-)\left(\exp\left\{iuf(T,H_s,z)\right\}-1-iuf(T,H_s,z)+\frac{u^2}{2}|f(T,H_s,z)|^2\right)\nu_0(dz)ds.
\end{align*}
For later use, we here present the following elementary inequality (cf. \cite{Dur10}): for all $u\in\mbbr$ and $n\in\mbbn\cup\{0\}$,
\begin{equation}\label{yu:ele}
\left|\exp(iu)-\sum_{j=0}^n\frac{(iu)^j}{j!}\right|\leq \frac{|u|^{n+1}}{(n+1)!}\wedge\frac{2|u|^n}{n!}.
\end{equation}
By the definition of $S$, we have $|\zeta(u,T)|\leq\exp\left\{u^2\Sig/2\right\}$.
Since $\int_0^T\int_\mbbr\zeta(u,s-)\left(\exp\left\{iuf(T,H_s,z)\right\}-1\right)\tilde{N}(ds,dz)$ is an $L_2$-martingale (cf. \cite[Section 4]{App09}) from these estimates, the optional sampling theorem implies that
\begin{equation*}
E\left[\int_0^{T\wedge S}\int_\mbbr\zeta(u,s-)\left(\exp\left\{iuf(T,H_s,z)\right\}-1\right)\tilde{N}(ds,dz)\right]=0.
\end{equation*} 
Next we show that
\begin{equation*}
E\left[\int_0^{T\wedge S}\int_\mbbr\zeta(u,s-)\left(\exp\left\{iuf(T,H_s,z)\right\}-1-iuf(T,H_s,z)+\frac{u^2}{2}|f(T,H_s,z)|^2\right)\nu_0(dz)ds\right]\to0.
\end{equation*}
Again using the above estimates, we have
\begin{align*}
&\left|E\left[\int_0^{T\wedge S}\int_\mbbr\zeta(u,s-)\left(\exp\left\{iuf(T,H_s,z)\right\}-1-iuf(T,H_s,z)+\frac{u^2}{2}|f(T,H_s,z)|^2\right)\nu_0(dz)ds\right]\right|\\
&\leq E\left[\int_0^{T\wedge S}\int_\mbbr\exp\left\{\frac{u^2}{2}\Sig\right\}\left(\frac{|uf(T,H_s,z)|^3}{6}\wedge|uf(T,H_s,z)|^2\right)\nu_0(dz)ds\right]\\
&\leq C_\delta\exp\left\{\frac{u^2}{2}\Sig\right\}E\left[\int_0^T\int_\mbbr |uf(T,H_s,z)|^{2+\delta}\nu_0(dz)ds\right]\to0,
\end{align*}
where $C_\delta$ is a positive constant such that 
\begin{equation*}
\frac{|x|^3}{6}\wedge|x|^2\leq C_\delta |x|^{2+\delta}
\end{equation*}
for all $x\in\mbbr$.
At last we observe that $X_1(T\wedge S)-X_1(T)\overset{P}\to0$.
In view of Lenglart's inequality and the isometry property of stochastic integral with respect to Poisson random measure (cf. \cite[Section 4]{App09}),
it suffices to show $E[\int_{T\wedge S}^T \int_\mbbr |f(T,H_s,z)|^2\nu_0(dz)ds]\to0$.
However the latter convergence is clear from Assumption (1).
Hence the proof is complete.
\end{proof}

\medskip
Next we show the following lemma which gives the fundamental small time moment estimate of $X$:
\begin{Lem}\label{FEV}
Under Assumptions \ref{Moments}-\ref{Stability}, it follows that
\begin{equation}\label{ME1}
E_{j-1}[|X_s-X_{j-1}|^p]\lesssim h_n (1+|X_{j-1}|^{p}),
\end{equation}
for any positive constant $p\in(1\vee\beta,2)$ and $s\in(t_{j-1},t_j]$.
\end{Lem}
\begin{proof}
Recall that $\int |z|^p\nu_0(dz)<\infty$ from Assumption \ref{Moments}.
By Lipschitz continuity of the coefficients and \cite[Theorem 1.1]{Fig08}, it follows that
\begin{align*}
&E_{j-1}[|X_s-X_{j-1}|^p]\\
&\lesssim E_{j-1}\left[\left|\int_{t_{j-1}}^s (A_{u}-A_{j-1})du+\int_{t_{j-1}}^s(C_{u-}-C_{j-1})dZ_{u}\right|^p+h_n^p |A_{j-1}|^p+h_n|C_{j-1}|^p\int_\mbbr |z|^p\nu_0(dz)\right]+o_p(h_n)\\
&\lesssim h_n\left(1+|X_{j-1}|^{p}+o_p(1)\right)+h_n^{p-1}\intj E_{j-1}[|X_s-X_{j-1}|^p] ds+E_{j-1}\left[\left|\intj(C_{s-}-C_{j-1})dZ_s\right|^p\right].
\end{align*}
Applying Burkholder-Davis-Gundy's inequality (cf. \cite[Theorem 48]{Pro04-2}), we have
\begin{align*}
E_{j-1}\left[\left|\intj(C_{s-}-C_{j-1})dZ_s\right|^p\right]&\lesssim E_{j-1}\left[\left(\intj\int_\mbbr (C_{s-}-C_{j-1})^2z^2N(ds,dz)\right)^{p/2}\right]\\
&= E_{j-1}\left[\left(\sum_{t_{j-1}\leq s<t_{j}} (C_{s-}-C_{j-1})^2(Z_s-Z_{s-})^2\right)^{p/2}\right]\\
&\leq E_{j-1}\left[\sum_{t_{j-1}\leq s<t_{j}} |C_{s-}-C_{j-1}|^p|Z_s-Z_{s-}|^p\right]\\
&= \intj E_{j-1}[|X_s-X_{j-1}|^p]ds\int_\mbbr |z|^p\nu_0(dz),
\end{align*}
for the Poisson random measure $N(ds,dz)$ associated with $Z$.
Hence Gronwall's inequality gives \eqref{ME1}.
\end{proof}

\medskip
\noindent\textbf{Proof of Theorem \ref{AN}}
\indent According to Cramer-Wold device, it is enough to show for $p_\gam=p_\al=1$ .
From a similar estimates used in Theorem \ref{TPE}, we have 
\begin{align}
\sqrt{T_n}\p_\gam\mbbg_{1,n}(\gam^\star)&=-\frac{2}{\sqrt{T_n}}\sumj\left\{\frac{\p_\gam c_{j-1}}{c^3_{j-1}}\left(h_nc^2_{j-1}-(\D_j X)^2\right)\right\}\nn\\
&=-\frac{2}{\sqrt{T_n}}\sumj\left\{\frac{\p_\gam c_{j-1}}{c^3_{j-1}}(h_nc^2_{j-1}-C^2_{j-1}(\D_j Z)^2)\right\}+o_p(1)\nn\\
&=-\frac{2}{\sqrt{T_n}}\sumj\left\{\frac{\p_\gam c_{j-1}}{c^3_{j-1}}C^2_{j-1}(h_n-(\D_j Z)^2)\right\}-\frac{2}{\sqrt{T_n}}\int_0^{T_n}\frac{\p_\gam c_s}{c^3_s}(c_s^2-C_s^2)ds\nn\\
&-\frac{2}{\sqrt{T_n}}\sumj\intj\left\{\frac{\p_\gam c_{j-1}}{c^3_{j-1}}(c_{j-1}^2-C_{j-1}^2)-\frac{\p_\gam c_s}{c^3_s}(c_s^2-C_s^2)\right\}ds+o_p(1)\nn\\
&=:\mbbf_{1,n}+\mbbf_{2,n}+\mbbf_{3,n}+o_p(1) \label{dec}.
\end{align}
We evaluate each term separately below.
Rewriting $\mbbf_{1,n}$ in a stochastic integral form via It\^{o}'s formula, we have
\begin{align*}
\mbbf_{1,n}&=-\frac{2}{\sqrt{T_n}}\sumj \intj\int_\mbbr \frac{\p_\gam c_{s-}}{c^3_{s-}}C_{s-}^2z^2\tilde{N}(ds,dz)-\frac{2}{\sqrt{T_n}}\sumj \intj\int_\mbbr \left(\frac{\p_\gam c_{j-1}}{c_{j-1}^3}C_{j-1}^2-\frac{\p_\gam c_{s-}}{c^3_{s-}}C_{s-}^2\right)z^2\tilde{N}(ds,dz)\\
&-\frac{4}{\sqrt{T_n}}\sumj\frac{\p_\gam c_{j-1}}{c_{j-1}^3}C_{j-1}^2\intj (Z_{s-}-Z_{j-1})dZ_s.
\end{align*}
for the compensated Poisson random measure $\tilde{N}(ds,dz)$ associated with $Z$.
Using Burkholder's inequality and the isometry property, it follows that for a positive constant $K$, 
\begin{align*}
&E\left[\left(\frac{1}{\sqrt{T_n}}\sumj \intj\int_\mbbr \left(\frac{\p_\gam c_{j-1}}{c_{j-1}^3}C_{j-1}^2-\frac{\p_\gam c_{s-}}{c^3_{s-}}C_{s-}^2\right)z^2\tilde{N}(ds,dz)\right)^2\right]\\
&\lesssim \frac{1}{T_n}\sumj E\left[\left(\intj\int_\mbbr \left(\frac{\p_\gam c_{j-1}}{c_{j-1}^3}C_{j-1}^2-\frac{\p_\gam c_{s-}}{c^3_{s-}}C_{s-}^2\right)z^2\tilde{N}(ds,dz)\right)^2\right]\\
&\lesssim \frac{1}{T_n}\sumj \intj E\left[\left(\int_0^1\p_x\left(\frac{\p_\gam c}{c^3}C^2\right)(X_{j-1}+u(X_s-X_{j-1}))du\right)(X_s-X_{j-1})\right]ds\\
&\lesssim \frac{1}{T_n}\sumj\intj \sqrt{\sup_{t\in\mbbr^+} E[1+|X_t|^K]}\sqrt{E[(X_s-X_{j-1})^2]}ds\\
&\lesssim \sqrt{h_n},
\end{align*}
and that
\begin{equation*}
E\left[\left|\intj (J_{s-}-J_{j-1})dJ_s\right|^2\right]\lesssim \intj E[|J_{s-t_{j-1}}|^2]ds\leq h_n^2.
\end{equation*}
Hence 
\begin{equation*}
\mbbf_{1,n}=-\frac{2}{\sqrt{T_n}}\sumj \intj\int_\mbbr \frac{\p_\gam c_{s-}}{c^3_{s-}}C_{s-}^2z^2\tilde{N}(ds,dz)+o_p(1).
\end{equation*}
Let us turn to observe $\mbbf_{2,n}$.
Let $f_{i,t}:=f_i(X_t)$ for $i=1,2$, and especially, let $f_{i,j}:=f_i(X_{t_j})$.
From Proposition \ref{EPE}, we obtain
\begin{align*}
\mbbf_{2,n}&=-\frac{2}{\sqrt{T_n}}\sumj \left(f_{1,j}- f_{1,j-1}+\intj \frac{\p_\gam c_s}{c^3_s}(c_s^2-C_s^2)ds\right)-\frac{2}{\sqrt{T_n}}(f_{1,n}-f_{1,0})\\
&=-\frac{2}{\sqrt{T_n}}\sumj \left(f_{1,j}- f_{1,j-1}+\intj \frac{\p_\gam c_s}{c^3_s}(c_s^2-C_s^2)ds\right)+o_p(1).
\end{align*}
For abbreviation, we simply write 
\begin{equation*}
\xi_{1,j}(t)=f_{1,t}- f_{1,j-1}+\int_{t_{j-1}}^t \p_\gam c_s(c_s^2-C_s^2)/c^3_sds.
\end{equation*}
According to Proposition \ref{EPE}, the weighted H\"{o}lder continuity of $f$, and Lemma \ref{FEV}, $\{\xi_{1,j}(t), \mcf_{t_{j-1}+t}:t\in[0,h_n]\}$ turns out to be an $L_2$-martingale.
Thus the martingale representation theorem \cite[Theorem I\hspace{-.1em}I\hspace{-.1em}I. 4. 34]{JacShi03} implies that there exists a predictable process $s\mapsto  \tilde{\xi}_{1,j}(s,z)$ such that
\begin{equation*}
\xi_{1,j}(t)=\int_{t_{j-1}}^t\int_\mbbr \tilde{\xi}_{1,j}(s,z)\tilde{N}(ds,dz).
\end{equation*}
Hence the continuous martingale component of $\xi_{1,j}$ is 0. 
By the property of $f_1$, we can define the stochastic integral $\int_{t_{j-1}}^t\int_\mbbr \left(f_1(X_{s-}+C_{s-}z)-f_1(X_{s-})\right)\tilde{N}(ds,dz)$ on $t\in[t_{j-1},t_j]$ and this process is also an $L_2$-martingale with respect to $\{\mcf_{t_{j-1}+t}:t\in[0,h_n]\}$.
Utilizing \cite[Theorem I. 4. 52]{JacShi03} and \cite[Corollary I\hspace{-.1em}I. 6. 3]{Pro04-2}, we have 
\begin{align*}
& E\left[\left|\frac{1}{\sqrt{T_n}}\sumj \left\{\xi_{1,j}(t_j)-\intj\int_\mbbr \left(f_1(X_{s-}+C_{s-}z)-f_1(X_{s-})\right)\tilde{N}(ds,dz) \right\}\right|^2\right]\\
& \lesssim \frac{1}{T_n}\sumj E\left[\left|\xi_{1,j}(t_j)-\intj\int_\mbbr \left(f_1(X_{s-}+C_{s-}z)-f_1(X_{s-})\right)\tilde{N}(ds,dz) \right|^2\right]\\
& = \frac{1}{T_n}\sumj E\left[\left[\xi_{1,j}(\cdot)-\int_{t_{j-1}}^\cdot\int_\mbbr \left(f_1(X_{s-}+C_{s-}z)-f_1(X_{s-})\right)\tilde{N}(ds,dz)\right]_{t_j}\right]=0.
\end{align*}
Here $[Y_\cdot]_t$ denotes the quadratic variation for any semimartingale $Y$ at time $t$, and we used Burkholder's inequality for a martingale difference between the first line and the second line.
By similar estimates above, we have $\mbbf_{3,n}=o_p(1)$.
Having these arguments in hand, it turns out that
\begin{equation*}
\sqrt{T_n}\p_\gam\mbbg_{1,n}(\gam^\star)=-\frac{2}{\sqrt{T_n}}\int_0^{T_n} \int_\mbbr \left(\frac{\p_\gam c_{s-}}{c^3_{s-}}C_{s-}^2z^2+f_1(X_{s-}+C_{s-}z)-f_1(X_{s-})\right)\tilde{N}(ds,dz)+o_p(1).
\end{equation*}
We can deduce from Assumption \ref{Smoothness} and Proposition \ref{EPE} that there exist positive constants $K,K',K''$ and $\ep_0<1\wedge(2-\beta)$ such that for all $z\in\mbbr$
\begin{align*}
&\sup_t \left\{\frac{1}{t}\int_0^t E\left[\left(\frac{\p_\gam c_{s}}{c^3_{s}}C_{s}^2z^2+f_1(X_{s}-C_{s}z)-f_1(X_{s})\right)^2\right]ds\right\}\\
&\lesssim \sup_t \left\{\frac{1}{t}\int_0^t \left(|z|^{2-\epsilon_0}\vee z^4\right)\left(1+\sup_t E\left[|X_t|^K\right]+\left(1+\sup_t E\left[|X_t|^{K'}\right]\right)|z|^{K''}\right)ds\right\}\\
&\lesssim (|z|^{2-\epsilon_0}\vee z^4)\left(1+|z|^{K''}\right),
\end{align*}
and the last term is $\nu_0$-integrable.
Then, there exist positive constants $K$ and $K'$ (possibly take different values from the previous ones) such that for any $z\in\mbbr$,
\begin{align*}
&\left|\frac{1}{t}\int_0^t E\left[\left(\frac{\p_\gam c_{s}}{c^3_{s}}C_{s}^2z^2+f_1(X_{s}+C_{s}z)-f_1(X_{s})\right)^2\right]ds\right.\\
&\qquad \qquad\qquad\left.-\int_\mbbr\left(\frac{\p_\gam c(y,\gam^\star)}{c^3(y,\gam^\star)}C^2(y)z^2+f_1(y+C(y)z)-f_1(y)\right)^2\pi_0(dy)\right|\\
&=\left| \frac{1}{t}\int_0^t\int_\mbbr \int_\mbbr \left(\frac{\p_\gam c(y,\gam^\star)}{c^3(y,\gam^\star)}C^2(y)z^2+f_1(y+C(y)z)-f_1(y)\right)^2  (P_s(x,dy)-\pi_0(dy)) \eta(dx)ds\right|\\
&\lesssim (|z|^{2-\epsilon_0}\vee z^4)\left(1+ |z|^{K'}\right) \frac{1}{t}\int_0^t\int_\mbbr ||P_s(x,\cdot)-\pi_0(\cdot)||_{h_K}\eta(dx)ds\\
&\to0.
\end{align*}
Thus the dominated convergence theorem and the isometry property give
\begin{align*}
&\lim_{n\to\infty} E\left[\left(\frac{1}{\sqrt{T_n}}\int_0^{T_n} \int_\mbbr \left(\frac{\p_\gam c_{s-}}{c^3_{s-}}C_{s-}^2z^2+f_1(X_{s-}+C_{s-}z)-f_1(X_{s-})\right)\tilde{N}(ds,dz)\right)^2\right]\\
&=\lim_{n\to\infty}\frac{1}{T_n}\int_0^{T_n} \int_\mbbr E\left[\left(\frac{\p_\gam c_{s}}{c^3_{s}}C_{s}^2z^2+f_1(X_{s}+C_{s}z)-f_1(X_{s})\right)^2\right]\nu_0(dz)ds\\
&=\frac{1}{4}\Sig_\gam.
\end{align*}
It follows from Assumption \ref{Stability} and Proposition \ref{EPE} that
\begin{equation*}
\lim_{n\to\infty}E\left[\int_0^{T_n}\int_\mbbr\left|\frac{1}{\sqrt{T_n}}\left(\frac{\p_\gam c_{s}}{c^3_{s}}C_{s}^2z^2+f_1(X_{s}+C_{s}z)-f_1(X_{s})\right)\right|^{2+K} \nu_0(dz)ds\right]\to0.
\end{equation*}
From Taylor expansion around $\gam^\star$, $\p_\al\mbbg_{2,n}(\al)$ is decomposed as:
\begin{align*}
\sqrt{T_n}\p_\al\mbbg_{2,n}(\al^\star)&=\frac{2}{\sqrt{T_n}}\sumj\frac{\p_\al a_{j-1}}{c^2_{j-1}}(\D_j X-h_na_{j-1})+\frac{2}{T_n}\sumj\p_\al a_{j-1}(\D_j X-h_na_{j-1}) \p_\gam c^{-2}_{j-1}\left(\sqrt{T_n}(\ges-\gam^\star)\right)\\
&+\left(\int_0^1\frac{2}{(T_n)^{3/2}}\sumj\p_\al a_{j-1}(\D_j X-h_na_{j-1}) \p_\gam^2 c^{-2}_{j-1}(\gam^\star+u(\ges-\gam^\star))du\right)\left(\sqrt{T_n}(\ges-\gam^\star)\right)^2
\end{align*}
Sobolev's inequality and the tail probability estimates of $\ges$ imply that the third term of the right-hand-side is $o_p(1)$.
Hence a similar manner to the first half leads to
\begin{align*}
&\sqrt{T_n}\p_\al\mbbg_{2,n}(\al^\star)-\frac{2}{T_n}\sumj\p_\al a_{j-1}(\D_j X-h_na_{j-1}) \p_\gam c^{-2}_{j-1}\left(\sqrt{T_n}(\ges-\gam^\star)\right) \\
&=\frac{2}{\sqrt{T_n}}\sumj\frac{\p_\al a_{j-1}}{c^2_{j-1}}(\D_j X-h_na_{j-1})+o_p(1)\\
&=\frac{2}{\sqrt{T_n}}\sumj\frac{\p_\al a_{j-1}}{c^2_{j-1}}C_{j-1}\D_j Z +\frac{2}{\sqrt{T_n}}\int_0^{T_n} \frac{\p_\al a_s}{c^2_s}(A_s-a_s)ds+o_p(1)\\
&=\frac{2}{\sqrt{T_n}}\sumj\intj\int_\mbbr\frac{\p_\al a_s}{c^2_{s-}}C_{s-}\tilde{N}(ds,dz)+\frac{2}{\sqrt{T_n}}\sumj\left(f_{2,j}-f_{2,j-1}+\intj \frac{\p_\al a_s}{c^2_s}(A_s-a_s)ds\right)+o_p(1)\\
&=\frac{2}{\sqrt{T_n}}\int_0^{T_n}\int_\mbbr\left(\frac{\p_\al a_s}{c^2_{s-}}C_{s-}z+f_2(X_{s-}+C_{s-}z)-f_2(X_{s-})\right)\tilde{N}(ds,dz)+o_p(1),
\end{align*}
and we have
\begin{align*}
&\lim_{n\to\infty}E\left[\left(\frac{1}{\sqrt{T_n}}\int_0^{T_n}\int_\mbbr\left(\frac{\p_\al a_s}{c^2_{s-}}C_{s-}z+f_2(X_{s-}+C_{s-}z)-f_2(X_{s-})\right)\tilde{N}(ds,dz)\right)^2\right]=\frac{1}{4}\Sig_\al,\\
&\lim_{n\to\infty}E\left[\int_0^{T_n}\int_\mbbr\left|\frac{1}{\sqrt{T_n}}\left(\frac{\p_\al a_s}{c^2_{s}}C_{s}z+f_2(X_{s}+C_{s}z)-f_2(X_{s})\right)\right|^{2+K}\nu_0(dz)ds\right]=0.
\end{align*}
From the isometry property and the trivial identity $xy=\left\{(x+y)^2-(x-y)^2\right\}/4$ for any $x,y\in\mbbr$, it follows that
\begin{align*}
\lim_{n\to\infty} &E\left[\left(\frac{1}{\sqrt{T_n}}\int_0^{T_n} \int_\mbbr \left(\frac{\p_\gam c_{s-}}{c^3_{s-}}C_{s-}^2z^2+f_1(X_{s-}+C_{s-}z)-f_1(X_{s-})\right)\tilde{N}(ds,dz)\right)\right.\\
&\left.\times\left(\frac{1}{\sqrt{T_n}}\int_0^{T_n}\int_\mbbr\left(\frac{\p_\al a_s}{c^2_{s-}}C_{s-}z+f_2(X_{s-}+C_{s-}z)-f_2(X_{s-})\right)\tilde{N}(ds,dz)\right)\right]=-\frac{1}{4}\Sig_{\al\gam}.
\end{align*}
Hence the moment estimates in the proof of Theorem \ref{TPE}, Lemma \ref{CLTPRM} and Taylor's formula yield that
\begin{align*}
&\ \quad\sqrt{T_n}\begin{pmatrix}-\p_\gam^2 \mbbg_{1,n}(\gam^\star)&0\\ -\frac{2}{T_n}\sumj\p_\al a_{j-1}(\D_j X-h_na_{j-1}) \p_\gam c^{-2}_{j-1}& -\p_\al^2 \mbbg_{2,n}(\al^\star)\end{pmatrix}\begin{pmatrix}\ges-\gam^\star \\ \aes-\al^\star\end{pmatrix}\\
&=\sqrt{T_n}\begin{pmatrix}\p_\gam \mbbg_{1,n}(\gam^\star)\\\p_\al \mbbg_{2,n}(\al^\star)\end{pmatrix}+o_p(1)\overset{\mcl}\longrightarrow N(0,\Sig).
\end{align*}
To achieve the desired result, it suffices to show 
\begin{align*}
\p_\gam^2 \mbbg_{1,n}(\gam^\star)\overset{P}\to \Gam_\gam,\ \p_\al^2 \mbbg_{2,n}(\al^\star)\overset{P}\to \Gam_\al,
\end{align*}
and
\begin{equation*}
 \frac{2}{T_n}\sumj\p_\al a_{j-1}(\D_j X-h_na_{j-1}) \p_\gam c^{-2}_{j-1}\overset{P}\to\Gam_{\al\gam}.
 \end{equation*}
However the first two convergence are straightforward from the proof of Theorem \ref{TPE}, and the last convergence follows from the ergodic theorem.
Thus the proof is complete.

\subsection*{Acknowledgement}
The author would like to thank Professor H. Masuda for his constructive comments.
He is also thank Professor A. M. Kulik for his advice about extended Poisson equation.
Finally, he is grateful to the anonymous referees for their valuable and constructive comments.
This work was supported by JST CREST Grant Number JPMJCR14D7, Japan.

\bibliographystyle{abbrv}

\end{document}